\documentclass[11pt,a4paper]{article}

\usepackage[utf8]{inputenc}     % Codificación de caracteres
\usepackage[T1]{fontenc}        % Codificación de salida
\usepackage[english]{babel}     % Idioma del documento
\usepackage{lineno,hyperref}
\usepackage[left=2cm,top=2.5cm,right=2cm,bottom=2.5cm]{geometry}
\usepackage{graphicx, tikz,array,delarray}
\usepackage{subfigure}
\usepackage{latexsym}
\usepackage{amscd}
\usepackage[utf8]{inputenc}
\usepackage{verbatim}
\usepackage{mathrsfs}
\usepackage{fancyhdr}
\usepackage{float}
\usepackage{palatino}
\usepackage{mathtools}
\usepackage{array}
\usepackage{hyperref}
\usepackage{authblk} % Para manejar autores y afiliaciones

\usepackage{amsmath,amssymb,amsthm}
\usepackage{mathtools}

\usepackage{geometry}
\usepackage[style=numeric,sorting=none,backend=biber]{biblatex}
\newtheorem{theorem}{\textbf{Theorem}}

\newtheorem{definition}[theorem]{\textbf{Definition}}

\newtheorem{proposition}[theorem]{\textbf{Proposition}}
\newtheorem{remark}[theorem]{Remark}

\title{A boundary integral approach to the eigenvalue problem for the anisotropic bidomain operator with perfect contact conditions}

\author[1]{Raul Felipe-Sosa\thanks{Corresponding author: \texttt{rfelipe@nebrija.es}}}
\author[2]{Yofre H. García-Gómez}

\affil[1]{Departamento de Matemáticas, Escuela Politénica Superior, Universidad Nebrija, 28015 Madrid, Spain}
\affil[2]{Faculty of Sciences in Physics and Mathematics, Universidad Autónoma de Chiapas (UNACH), Chiapas, México}

\begin{document}

\maketitle

\begin{abstract}
In this work, we study the eigenvalue problem associated with the bidomain operator in an anisotropic heterogeneous domain composed of three subregions representing the left ventricle, the septum, and the right ventricle. The anisotropic conductivity, together with the different orientations of the fiber directions in each subdomain, leads to an elliptic boundary value problem with discontinuous coefficients and transmission conditions across the interfaces.

Our main contribution consists in reformulating this problem using potential theory. By expressing the solution in terms of single- and double-layer potentials, we reduce the original boundary value problem to a system of Fredholm-type boundary integral equations. We derive explicit expressions for the fundamental solution of the associated anisotropic Helmholtz operator, as well as for the corresponding kernels, which are given in terms of Bessel functions.

Finally, we propose a numerical scheme for approximating the eigenvalues of the bidomain operator based on the discretization of the resulting integral system. This approach provides an efficient framework for the analysis of anisotropic boundary value problems with interface conditions.
\end{abstract}


\section{Introduction}
The mechanical activity of the heart is driven by the propagation of bioelectrical impulses through the myocardium. At the microscopic level, this process is governed by the ability of cardiac cells to become electrically activated through the exchange of ions across their membranes. This phenomenon, which is essential for the synchronized contraction of the cardiac chambers, can be mathematically described by a system of nonlinear partial differential equations defined on a domain that exhibits anisotropic electrical conductivity, governing the evolution of the electrical potential in the myocardium.

One of the most widely recognized frameworks for modeling cardiac electrical activity is the \emph{bidomain model}, which represents cardiac tissue as a medium composed of two superimposed conductive domains, namely the intracellular and extracellular domains, coupled through the cell membrane. In its modern formulation, the model was first introduced by L.~Tung in his doctoral dissertation~\cite{Tung1978}. For a comprehensive analysis of this model, the reader is referred to~\cite{Sundnes2006}. The bidomain model consists of a system of nonlinear partial differential equations involving multiple types of equations---parabolic, elliptic, and ordinary---which significantly increase the complexity of both its theoretical analysis and numerical simulation.

From a theoretical perspective, mathematicians have investigated, on the one hand, the well-posedness of the associated boundary value problem, which involves the proper formulation of the problem and the proof of existence and uniqueness of solutions within this framework; see \cite{Colli-Franzone2002}, \cite{Veneroni2009}, and \cite{Bourgault2009}. On the other hand, several results have been established regarding the existence of periodic solutions when periodic activation of the endocardium is imposed as a boundary condition in the model, as discussed in \cite{Felipe-Sosa2022} and \cite{Felipe-Sosa2023}.

The numerical analysis of the bidomain model faces significant challenges, including the complex cardiac geometry and the anisotropy of electrical current propagation \cite{Sundnes2006}. Implementing a numerical algorithm requires the discretization of both the cardiac geometry and the anisotropy tensor, which is closely related to the physiological properties and structural organization of the heart. For instance, in \cite{Boulakia2010}, numerical solutions of the bidomain model were obtained using only the ventricular geometry, approximated by ellipsoids of revolution. Within this framework, the electrical conductivity tensor was aligned with the local fiber direction. Nevertheless, the results revealed several noteworthy features; for example, a portion of an ECG was simulated with a high degree of realism. Additional works on numerical simulations of cardiac electrical activity include \cite{Colli-Franzone2006}, \cite{Patelli2017}, \cite{Bucelli2021}, and \cite{Charawi2017}.

Despite the extensive body of work devoted to the numerical approximation of the bidomain model, most existing approaches rely on volumetric discretization techniques, such as finite element methods, which require the resolution of large-scale systems defined on complex geometries. These methods typically involve a fine discretization of the entire computational domain in order to accurately capture both the anisotropic properties of the tissue and the intricate cardiac geometry, leading to a significant computational burden.

In contrast, alternative formulations based on boundary integral methods remain largely unexplored in this context. Such approaches offer the potential advantage of reducing the dimensionality of the problem, as they reformulate the governing equations in terms of quantities defined on the interfaces between subdomains. This reduction is particularly appealing in the present setting, where the discontinuities in the conductivity tensors are localized at the interfaces between the ventricles and the septum.

An interesting problem that can be addressed using the bidomain model is to understand why the septum acts as a functional barrier to the wavefront. When both ventricles are activated, the wavefront originating in the left ventricle fails to propagate effectively through the septum into the right ventricle, and vice versa. This occurs despite the fact that the septum lacks anatomical features that distinguish it from the ventricles. This phenomenon was reported in \cite{Medrano2002} and \cite{Medrano1957}, where experimental observations showed that, in the presence of bundle branch block---that is, when one of the ventricles is not activated---the wavefront from the active ventricle can propagate into the other, although at a significantly lower velocity than expected. This suggests that, although current diffuses through the septal tissue, there may exist a discontinuity in the electrical conductivity tensors at the interface between the septum and the ventricles. To the best of our knowledge, this hypothesis has not been previously proposed or examined, either in the medical literature or in mathematical modeling, and constitutes one of the novel contributions of this work.

These considerations motivate the study of the bidomain model in a three-phase region, where each phase corresponds to the left ventricle, the septum, and the right ventricle, respectively. In this setting, we assume that the conductivity tensors exhibit a discontinuity at the left ventricle--septum and right ventricle--septum interfaces, the specific nature of which will be described later in this work.

In \cite{Felipe-Sosa2023}, it was shown that a weak $T$-periodic solution exists for the bidomain problem coupled with the torso, obtained as the limit of a Faedo--Galerkin sequence constructed from the eigenvectors of the bidomain operator. This highlights the importance of the eigenvectors, not only from a theoretical standpoint but also as a tool for numerical approximation. For this reason, the present study focuses on the computation of these eigenvectors within the context described above.

In summary, this work develops a methodology to solve the eigenvalue problem associated with the bidomain operator in an anisotropic region that idealizes the left ventricle, the septum, and the right ventricle of the heart. The approach is based on potential theory, representing the solutions as linear combinations of single- and double-layer potentials. In Chapter~1 of \cite{Linkov2002}, the fundamental aspects of this theory are presented in the context of planar linear elasticity. Here, we adapt this approach to the problem under consideration. In this way, the original problem is reduced to a Fredholm integral equation of the second kind posed on the interfaces, whose numerical solution yields the desired eigenvalues.

The article is organized as follows. Section~2 presents the mathematical formulation of the boundary value problem and describes the spatial region where it is defined. In addition, the electrical conductivity tensors are constructed according to the type of anisotropy under consideration. In Section~3, we apply potential theory to the anisotropic problem defined in the previous section and compute the corresponding fundamental solution. Section~4 reduces the boundary value problem to a system of singular integral equations, whose solution yields the single- and double-layer potentials involved in the solution of the boundary value problem under consideration. Finally, Section~5 presents the numerical solution of the system of integral equations, thereby obtaining a numerical solution of the boundary value problem.

\section{Boundary Value Problem Formulation and Conductivity Anisotropy}

In this work, we consider a two-dimensional spatial domain. Since our objective is to apply potential theory to the boundary value problem, and as a preliminary step, it is convenient to work in two dimensions. From a modeling perspective, this assumption does not appear to be restrictive. Specifically, the domain represents an infinitesimal cross-section of the myocardium centered at the interventricular septum, including tissue from the left and right ventricles, as well as the septum. This configuration allows us to capture variations in the longitudinal fiber directions, giving rise to jump discontinuities at the interfaces between the three regions.

This section is organized into four parts. In the first part, we formulate the eigenvalue boundary value problem in the region of interest, specifying the governing equations, the boundary conditions, and the transmission conditions that arise at the interfaces between the three subregions under consideration.

In the second part, we describe the spatial domain and the type of anisotropy it exhibits.

The third part introduces what we refer to as the direct formulation, in which we describe how the solutions of the equation can be expressed as a combination of single- and double-layer potentials. In this context, the densities correspond to the values of the electrical potentials and currents on the outer boundary and at the interfaces between the subdomains representing the left ventricle, the septum, and the right ventricle.

Finally, we derive the fundamental solution associated with the governing equation, noting that this solution defines the kernel of the single-layer potential, while its normal derivative defines the kernel of the double-layer potential.

\subsection{Boundary Value Problem Formulation}

We denote the domain by $\Omega = \Omega_I \cup \Omega_S \cup \Omega_D$, where $\Omega_I$, $\Omega_D$, and $\Omega_S$ represent the right ventricle, the left ventricle, and the septum, respectively. More precisely, these subdomains are defined as follows:
\begin{align*}
\Omega_I &= \left(0, \frac{a}{2} - h\right)\times (0, b),\\
\Omega_S &= \left(\frac{a}{2} - h, \frac{a}{2} + h\right)\times (0, b),\\
\Omega_D &= \left(\frac{a}{2} + h, a\right)\times (0, b).
\end{align*}
These subdomains are illustrated in Figure~\ref{figura1}.
% \begin{figure}[H]
% \centering
% \includegraphics[scale=0.5]{Figura1.jpg} 
% \caption{Geometric representation of the $\Omega$ region}
% \end{figure}

\begin{figure}[H]
\centering
\begin{tikzpicture}
    % Dibuja los ejes
    \draw[->] (-1,0) -- (7,0) node[right] {$x$}; % Eje X
    \draw[->] (0,-1) -- (0,4) node[above] {$y$}; % Eje Y

    % Dibuja el rectángulo
    \draw[thick] (0,0) rectangle (6,3);
    
    % Dibuja las líneas divisorias (en 3 partes iguales)
    \draw[thick] (2,0) -- (2,3); % Línea divisoria vertical 1
    \draw[thick] (4,0) -- (4,3); % Línea divisoria vertical 2

    % Etiquetas en el centro de cada subregión
    \node at (1, 1.5) {$\Omega_I$}; % Centro de la primera subregión
    \node at (3, 1.5) {$\Omega_S$}; % Centro de la segunda subregión
    \node at (5, 1.5) {$\Omega_D$}; % Centro de la tercera subregión

    % Etiquetas en las intersecciones del eje X
    \node at (2, -0.23) {$\frac{a}{2} - h$};   % Intersección en x = 2
    \node at (4, -0.23) {$\frac{a}{2} + h$};   % Intersección en x = 4
    \node at (6, -0.23) {$a$};   % Intersección en x = 6   

    % Etiquetas en las intersecciones del eje Y
    \node at (-0.2, 3) {$b$};   % Intersección en y = 3

    % Cambia el color del segmento c-d a rojo
    \draw[ thick] (4, 0) -- (6, 0); % Segmento rojo entre c y d en el eje X
    \draw[ thick] (2, 0) -- (4, 0); 
    \draw[ thick] (2, 3) -- (4, 3); 
    \draw[ thick] (2,0) -- (2,3); % Línea divisoria vertical 1
    \draw[ thick] (4,0) -- (4,3); % Línea divisoria vertical 2
    \draw[ thick] (0, 0) -- (2, 0); 
    \draw[thick] (0,0) -- (0,3); % Línea vertical sobre eje X
    \draw[thick] (0,3) -- (2,3); % Línea divisoria vertical 2
    \draw[thick] (6,0) -- (6,3); % Línea divisoria vertical 2
    \draw[thick] (4,3) -- (6,3); % Línea divisoria vertical 1
       
\end{tikzpicture}
\caption{Geometric representation of the $\Omega$ region}
\label{figura1}
\end{figure}
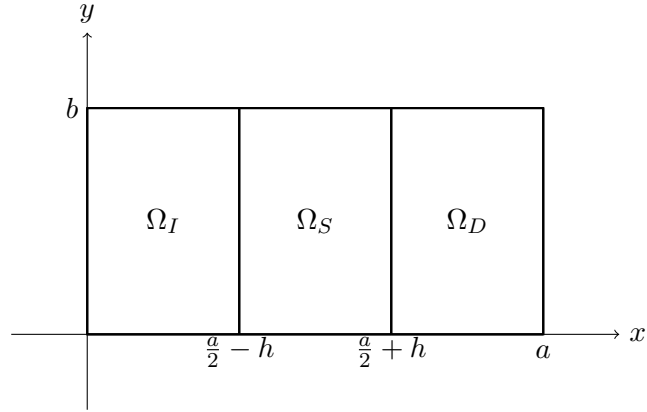

As we will see later, it is convenient to introduce the following notation:
\begin{itemize}
\item We define $\partial \Omega_I = \Gamma_I \cup \Gamma_{IS}$, where
\begin{align*}
\Gamma_I &= \left\{(x_1, b): x_1 \in \left[0, \tfrac{a}{2} - h\right]\right\} 
\cup \left\{(0, x_2): x_2 \in [0, b]\right\} 
\cup \left\{(x_1, 0): x_1 \in \left[0, \tfrac{a}{2} - h\right]\right\},
\end{align*}
is the outer boundary of $\Omega_I$, and
\begin{align*}
\Gamma_{IS} &= \left\{\left(\tfrac{a}{2} - h, x_2\right): x_2 \in [0, b]\right\},
\end{align*}
is the interface between $\Omega_I$ and $\Omega_S$.

\item We define $\partial \Omega_S = \Gamma_{IS} \cup \Gamma_S \cup \Gamma_{SD}$, where
\begin{align*}
\Gamma_S &= \left\{(x_1, b): x_1 \in \left[\tfrac{a}{2} - h, \tfrac{a}{2} + h\right]\right\}
\cup \left\{(x_1, 0): x_1 \in \left[\tfrac{a}{2} - h, \tfrac{a}{2} + h\right]\right\},
\end{align*}
is the outer boundary of $\Omega_S$, and
\begin{align*}
\Gamma_{SD} &= \left\{\left(\tfrac{a}{2} + h, x_2\right): x_2 \in [0, b]\right\},
\end{align*}
is the interface between $\Omega_S$ and $\Omega_D$.

\item Finally, we define $\partial \Omega_D = \Gamma_D \cup \Gamma_{SD}$, where
\begin{align*}
\Gamma_D &= \left\{(x_1, b): x_1 \in \left[\tfrac{a}{2} + h, a\right]\right\}
\cup \left\{(a, x_2): x_2 \in [0, b]\right\}
\cup \left\{(x_1, 0): x_1 \in \left[\tfrac{a}{2} + h, a\right]\right\},
\end{align*}
is the outer boundary of $\Omega_D$.

\item We also denote the external boundary of the domain by
\[
\partial_{\mathrm{ext}}\Omega = \Gamma_I \cup \Gamma_S \cup \Gamma_D.
\]
\end{itemize}

The eigenvalue problem associated with the bidomain operator consists in solving the following linear system in $\Omega$:
\begin{align}\label{Eqn1}
\left\{
\begin{array}{l}
\nabla \cdot\left(\sigma_i(\mathbf{x}) \nabla u_i\right) + \lambda^2 u_i = 0, \\
\nabla \cdot\left(\sigma_e(\mathbf{x}) \nabla u_e\right) + \lambda^2 u_e = 0,
\end{array}
\right. \qquad \mathbf{x} \in \Omega,
\end{align}
where $\mathbf{x} = (x_1, x_2)$ is a vector in $\mathbb{R}^2$, the operators $\nabla$ and $\nabla \cdot$ denote the gradient and divergence, respectively, $u_i$ and $u_e$ are the intra- and extracellular potentials, and $\sigma_i$, $\sigma_e$ are the corresponding intra- and extracellular conductivity tensors. 

Originally, the bidomain model can be formulated as a system of parabolic equations coupled with a system of ordinary differential equations that describe cellular activation. The nonlinear terms in the system are responsible for coupling the equations governing the intracellular and extracellular currents. However, as can be seen in \eqref{Eqn1}, if only the diffusion operator is considered, these equations decouple. Therefore, our methodology can be applied to each of these equations separately. Moreover, assuming that the heart is uncoupled from the torso, both the intracellular and extracellular potentials satisfy the same boundary and transmission conditions. Accordingly, we consider a partial differential equation of the following form:

\begin{align}\ref{Eqn2}
\nabla \cdot\left(\sigma(\mathbf{x}) \nabla u\right) + \lambda^2 u = 0,  
\end{align}
where the subscripts $i$ and $e$ have been omitted. In what follows, no distinction will be made, as both the boundary value problem formulation and the proposed solution methodology apply equally to the intracellular and extracellular currents. 

We also impose Neumann-type boundary conditions on the external boundary of $\Omega$, together with contact conditions at the interfaces between the subdomains composing the region. From a physiological point of view, this means that both the intracellular and extracellular currents are confined within the boundary of $\Omega$, and that both the potentials and the current fluxes remain continuous across the interfaces between subregions. Note that imposing a zero normal flux condition for the intracellular current implies that the heart is electrically isolated from the torso.

Before specifying the boundary and perfect contact conditions, we introduce the following notation:
\begin{align}\label{p_ie}
p(\mathbf{x}) = \nabla\left(\sigma(\mathbf{x}) \nabla u(\mathbf{x})\right) \cdot \mathbf{n}(\mathbf{x}), \quad \text{defined on } \partial \Omega \text{ and } \Gamma_{IS}, \Gamma_{SD},
\end{align}
which represents the intracellular and extracellular current fluxes across the corresponding boundaries. Here, $\mathbf{n}(\mathbf{x})$ denotes the outward unit normal vector at the boundary point $\mathbf{x}$.

The boundary conditions corresponding to zero normal current flux are given by:
\begin{align}\label{BC1}
p(\mathbf{x}) = 0, \quad \text{on } \partial_{\text{ext}} \Omega.
\end{align}

We also impose so-called perfect transmission conditions:
\begin{align}\label{PCC1}
\begin{cases}
p^{(I)}(\mathbf{x}) = p^{(S)}(\mathbf{x}), \\
u^{(I)}(\mathbf{x}) = u^{(S)}(\mathbf{x}),
\end{cases}
\quad \text{for } \mathbf{x} \in \Gamma_{IS},
\end{align}
\begin{align}\label{PCC2}
\begin{cases}
p^{(S)}(\mathbf{x}) = p^{(D)}(\mathbf{x}), \\
u^{(S)}(\mathbf{x}) = u^{(D)}(\mathbf{x}),
\end{cases}
\quad \text{for } \mathbf{x} \in \Gamma_{SD}.
\end{align}

Here, the superscripts denote traces taken from the corresponding subdomains. More precisely, for $\mathbf{x}\in\Gamma_{IS},$
\begin{align*}
\begin{cases}
p^{(I)}(\mathbf{x}) = \displaystyle\lim_{\xi \to \mathbf{x},\, \xi \in \Omega_I} p(\xi), \\
p^{(S)}(\mathbf{x}) = \displaystyle\lim_{\xi \to \mathbf{x},\, \xi \in \Omega_S} p(\xi), \\
u^{(I)}(\mathbf{x}) = \displaystyle\lim_{\xi \to \mathbf{x},\, \xi \in \Omega_I} u(\xi), \\
u^{(S)}(\mathbf{x}) = \displaystyle\lim_{\xi \to \mathbf{x},\, \xi \in \Omega_S} u(\xi),
\end{cases}
\end{align*}
and, for $\mathbf{x}\in\Gamma_{SD},$
\begin{align*}
\begin{cases}
p^{(S)}(\mathbf{x}) = \displaystyle\lim_{\xi \to \mathbf{x},\, \xi \in \Omega_S} p(\xi), \\
p^{(D)}(\mathbf{x}) = \displaystyle\lim_{\xi \to \mathbf{x},\, \xi \in \Omega_D} p(\xi), \\
u^{(S)}(\mathbf{x}) = \displaystyle\lim_{\xi \to \mathbf{x},\, \xi \in \Omega_S} u(\xi), \\
u^{(D)}(\mathbf{x}) = \displaystyle\lim_{\xi \to \mathbf{x},\, \xi \in \Omega_D} u(\xi).
\end{cases}
\end{align*}

Once the perfect transmission conditions have been established, it is worth emphasizing that they imply that, despite the discontinuities in the electrical conductivity tensors at the interfaces, the current remains continuous across them. However, from both a physiological and a mathematical perspective, it may be of interest to consider situations in which the current also exhibits jump discontinuities across these interfaces, leading to what we refer to as imperfect transmission conditions. Incorporating such conditions would introduce an additional level of complexity to the boundary value problem under consideration and would further motivate the use of potential theory for its resolution. This problem will be addressed in future work. 

\subsection{Electrical Conductivity Tensors and Anisotropic Structure}

Before proceeding with the solution of the eigenvalue problem, we define the type of anisotropy under consideration.

It is well known that electrical current in the heart exhibits preferred directions along which it propagates more rapidly. This is due to the fact that cardiomyocytes (cardiac cells) are elongated and arranged in an organized manner, which at the macroscopic level gives rise to a fiber-like structure. See \cite{Kotadia2020} for further details on this topic.

This anisotropy determines the structure of the tensor $\sigma$. Following \cite{Colli-Franzone2006}, we assume that at each point $\mathbf{x} \in \Omega$ there exist three orthonormal directions $\mathbf{a}_l(\mathbf{x}), \mathbf{a}_t(\mathbf{x}), \mathbf{a}_n(\mathbf{x})$, corresponding to the longitudinal, transverse, and normal directions to the fiber structure. Moreover, we assume that the longitudinal and transverse directions lie in the same plane. To construct the conductivity tensors, it is therefore necessary to specify how these directions are distributed throughout $\Omega$. In what follows, we assume that these directions are homogeneous within each phase; that is, the directions $\mathbf{a}_l(\mathbf{x})$ and $\mathbf{a}_t(\mathbf{x})$ are constant within each subregion of $\Omega$.

In this case, we assume that the longitudinal direction of the fibers $\mathbf{a}_l(x,y)$ is constant within each subregion of $\Omega$ and is defined as follows:
\begin{align*}
\mathbf{a}_l(x,y) =
\begin{cases}
(1, 0), & \text{if } (x,y) \in \Omega_I,\\
(0, 1), & \text{if } (x,y) \in \Omega_S,\\
(-1, 0), & \text{if } (x,y) \in \Omega_D.
\end{cases}
\end{align*}

Consequently, the transverse direction $\mathbf{a}_t(x,y)$ is orthogonal to $\mathbf{a}_l(x,y)$ and given by:
\begin{align*}
\mathbf{a}_t(x,y) =
\begin{cases}
(0, 1), & \text{if } (x,y) \in \Omega_I,\\
(1, 0), & \text{if } (x,y) \in \Omega_S,\\
(0, 1), & \text{if } (x,y) \in \Omega_D.
\end{cases}
\end{align*}

\begin{figure}[H]
\centering
\includegraphics[width=0.6\linewidth]{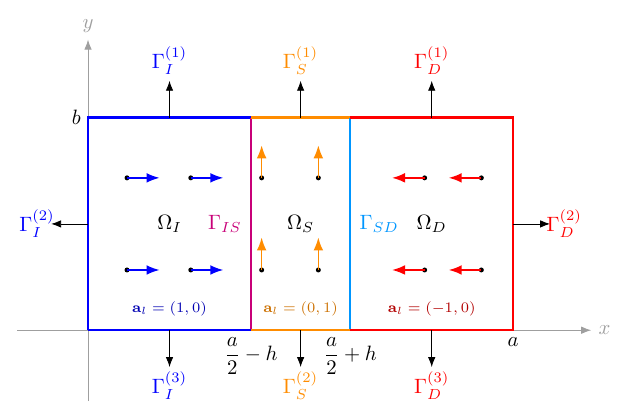}
\caption{Geometric representation of the longitudinal fiber direction}
%\label{fig:enter-label}
\end{figure}
Once the longitudinal and transverse fiber directions have been defined in each subregion of $\Omega$, the conductivity tensors can be constructed using the following formula:
\begin{align*}
\sigma(x_1,x_2) = \sigma_{l}\,\mathbf{a}_l(x_1,x_2)^{T}\mathbf{a}_l(x_1,x_2) + \sigma_{t}\,\mathbf{a}_t(x_1,x_2)^{T}\mathbf{a}_t(x_1,x_2),
\end{align*}
which yields
\begin{align}
\sigma(x_1,x_2) = 
\begin{cases}
\begin{pmatrix}
\sigma_l & 0\\
0 & \sigma_t
\end{pmatrix}, & \text{if } (x_1,x_2) \in \Omega_I, \\
\\
\begin{pmatrix}
\sigma_t & 0\\
0 & \sigma_l
\end{pmatrix}, & \text{if } (x_1,x_2) \in \Omega_S, \\
\\
\begin{pmatrix}
\sigma_l & 0\\
0 & \sigma_t
\end{pmatrix}, & \text{if } (x_1,x_2) \in \Omega_D.
\end{cases}
\end{align}

Here, $\sigma_l$ and $\sigma_t$ denote the local conductivities in the longitudinal and transverse directions with respect to the fiber orientation, respectively. In this study, we assume that $\sigma_l > \sigma_t$. It is important to note that the tensors $\sigma(x_1, x_2)$ are piecewise constant in $\Omega$. Consequently, when formulating the boundary value problem, it is necessary to impose appropriate contact conditions at the interfaces between the different subregions of $\Omega$. In this way, we aim to capture the experimentally observed behavior, discussed in the Introduction, whereby the septum acts as a functional barrier to the propagation of the wavefront.

\subsection{Direct formulation}\label{DirForSec}
In this section, we assume that both the electrical potential and the conductivity tensor are defined piecewise:
\[
u(\mathbf{x}) =
\begin{cases}
u^{(I)}(\mathbf{x}), & \text{if } \mathbf{x} \in \Omega_I,\\[0.1cm]
u^{(S)}(\mathbf{x}), & \text{if } \mathbf{x} \in \Omega_S,\\[0.1cm]
u^{(D)}(\mathbf{x}), & \text{if } \mathbf{x} \in \Omega_D,
\end{cases}
\]
and
\[
\sigma(\mathbf{x}) =
\begin{cases}
\sigma^{(I)}(\mathbf{x}), & \text{if } \mathbf{x} \in \Omega_I,\\[0.1cm]
\sigma^{(S)}(\mathbf{x}), & \text{if } \mathbf{x} \in \Omega_S,\\[0.1cm]
\sigma^{(D)}(\mathbf{x}), & \text{if } \mathbf{x} \in \Omega_D.
\end{cases}
\]

To develop the direct formulation of the problem in terms of single- and double-layer potentials, we focus on the dynamics within each region $\Omega_j$, with $j = I, S, D$. To fix ideas, we begin by deriving the direct formulation for the electrical potential in the region $\Omega_I$, and subsequently obtain the corresponding expressions in the other regions, which follow from a similar procedure. Throughout the following discussion, we adopt the notation $\mathbf{x} = (x_1, x_2)$ and $\mathbf{y} = (y_1, y_2)$ to denote points in $\mathbb{R}^2$.

We denote by $u^{(I)}(\mathbf{x})$ the electrical potential in $\Omega_I$, and by $u^{(*, I)}(\mathbf{x}, \mathbf{y}; \lambda)$ the fundamental solution associated with the anisotropic Helmholtz operator, which will be defined and explicitly derived in the next section. We emphasize that the fundamental solution depends on the eigenvalue $\lambda$.

Next, we multiply \eqref{Eqn1} by $u^{*(I)}(\mathbf{x}, \mathbf{y}; \lambda)$ and integrate over $\Omega_I$:

\begin{align}\label{formInt1}
\int_{\Omega_I} \left( \nabla \cdot \left( \sigma^{(I)}(\mathbf{y}) \nabla u^{(I)}_i(\mathbf{y}) \right) + \lambda^2 u^{(I)}_i(\mathbf{y}) \right) u^{(*,I)}_i(\mathbf{x}, \mathbf{y}; \lambda) \, d\Omega_I(\mathbf{y}) = 0,
\end{align}
for all $\mathbf{x} \in \Omega_I$. 

According to the definition of a fundamental solution (see Definition~1.3.5 in \cite{Ortner2015}), it follows that, in the sense of distributions,
\begin{align*}
\nabla \cdot \left( \sigma^{(I)}(\mathbf{x}) \nabla u^{*(I)}(\mathbf{x}, \mathbf{y}; \lambda) \right) + \lambda^2 u^{*(I)}(\mathbf{x}, \mathbf{y}; \lambda) = \delta(\mathbf{x} - \mathbf{y}),
\end{align*}
where $\delta(\mathbf{x} - \mathbf{y})$ denotes the Dirac delta distribution centered at $\mathbf{y}$.

In fact, it is sufficient to seek the fundamental solution in the form of a radial function, that is, a function of the form $u^{*(I)}\!\left(|\mathbf{x} - \mathbf{y}|; \lambda\right)$ satisfying, in the sense of distributions,
\begin{align*}
\nabla \cdot \left( \sigma^{(I)}(\mathbf{x}) \nabla u^{*(I)}(\mathbf{x}, \mathbf{y}; \lambda) \right) + \lambda^2 u^{*(I)}(\mathbf{x}, \mathbf{y}; \lambda) = \delta(\mathbf{x} - \mathbf{y}).
\end{align*}

Now, integrating by parts the expression given in \eqref{formInt1}, we obtain:
\begin{align*}
&\int_{\Omega_I} u^{(I)}(\mathbf{y}) \left( \nabla \cdot \left( \sigma^{(I)}(\mathbf{y}) \nabla u^{(*,I)}(\mathbf{x}, \mathbf{y}; \lambda) \right) + \lambda^2 u^{(*,I)}(\mathbf{x}, \mathbf{y}; \lambda) \right) \, d\Omega_I(\mathbf{y}) 
= \int_{\Gamma_I} u^{(I)}(\mathbf{y}) \, p^{(*,I)}(\mathbf{x}, \mathbf{y}; \lambda) \, d\Gamma_I(\mathbf{y}) \\
& - \int_{\Gamma_{IS}} p^{(I)}(\mathbf{y}) \, u^{(*,I)}(\mathbf{x}, \mathbf{y}; \lambda) \, d\Gamma_{IS}(\mathbf{y}).
\end{align*}
where
\[
p^{(*,I)}(\mathbf{x}, \mathbf{y}; \lambda)
= \sigma^{(I)}(\mathbf{y}) \nabla_{\mathbf{y}} u^{(*,I)}(\mathbf{x}, \mathbf{y}; \lambda)\cdot \mathbf{n}(\mathbf{y}),
\]
where $\mathbf{n}(\mathbf{y})$ is the outward unit normal vector at $\mathbf{y}$.

Since $u^{(*,I)}(\mathbf{x}, \mathbf{y}; \lambda)$ is the fundamental solution, the previous expression can be rewritten as follows:
\begin{align*}
\int_{\Omega_I} u^{(I)}(\mathbf{y}) \, \delta(\mathbf{x} - \mathbf{y}) \, d\Omega_I(\mathbf{y}) 
&= \int_{\Gamma_I} u^{(I)}(\mathbf{y}) \, p^{(*,I)}(\mathbf{x}, \mathbf{y}; \lambda) \, d\Gamma_I(\mathbf{y}) 
- \int_{\Gamma_{IS}} p^{(I)}(\mathbf{y}) \, u^{(*,I)}(\mathbf{x}, \mathbf{y}; \lambda) \, d\Gamma_{IS}(\mathbf{y}),
\end{align*}
and therefore, we obtain the following representation formula:
\begin{align}\label{repreuieI}
u^{(I)}(\mathbf{x}; \lambda) = \int_{\Gamma_I} u^{(I)}(\mathbf{y}) \, p^{(*,I)}(\mathbf{x}, \mathbf{y}; \lambda) \, d\Gamma_I(\mathbf{y}) 
- \int_{\Gamma_{IS}} p^{(I)}(\mathbf{y}) \, u^{(*,I)}(\mathbf{x}, \mathbf{y}; \lambda) \, d\Gamma_{IS}(\mathbf{y}),
\end{align}
for all $\mathbf{x} \in \Omega_I$. Note that the dependence of the electrical potential on $\lambda$ has been made explicit.

Thus, we have represented the electrical potential as a linear combination of single- and double-layer potentials. Within this framework, it is essential to determine the values of these potentials and the corresponding currents on the boundary of $\Omega_I$. However, not all boundary values can be obtained using only the boundary conditions in \eqref{BC1} and the perfect transmission conditions in \eqref{PCC1}--\eqref{PCC2}. In particular, the potential $u^{(I)}$ remains unknown on $\Gamma_I \setminus \Gamma_{IS}$, and on $\Gamma_{IS}$ we impose only the continuity of both the potentials and the currents. As a consequence, these values cannot be determined without also considering the behavior of the solution in the other regions, $\Omega_S$ and $\Omega_D$.

It is worth emphasizing that the superscript $(I)$ simply indicates that the analysis is restricted to the region $\Omega_I$. However, the method used to derive the representation formula \eqref{repreuieI} does not depend on the specific structure of the conductivity tensor $\sigma^{(I)}$. Consequently, analogous representations can be obtained for the potentials in the regions $\Omega_S$ and $\Omega_D$.

We have
\begin{align}\label{repreuieS}
u^{(S)}(\mathbf{x}; \lambda) &= \int_{\Gamma_S} u^{(S)}(\mathbf{y}) \, p^{(*,S)}(\mathbf{x}, \mathbf{y}; \lambda) \, d\Gamma_S(\mathbf{y}) 
- \int_{\Gamma_{IS}} p^{(S)}(\mathbf{y}) \, u^{(*,S)}(\mathbf{x}, \mathbf{y}; \lambda) \, d\Gamma_{IS}(\mathbf{y}) \\
\nonumber
&\quad - \int_{\Gamma_{SD}} p^{(S)}(\mathbf{y}) \, u^{(*,S)}(\mathbf{x}, \mathbf{y}; \lambda) \, d\Gamma_{SD}(\mathbf{y}),
\end{align}
for all $\mathbf{x} \in \Omega_S$.

Similarly, the representation for $u_{i,e}^{(D)}$ in $\Omega_D$ is given by:
\begin{align}\label{repreuieD}
u^{(D)}(\mathbf{x}; \lambda) &= \int_{\Gamma_D} u^{(D)}(\mathbf{y}) \, p^{(*,D)}(\mathbf{x}, \mathbf{y}; \lambda) \, d\Gamma_D(\mathbf{y}) 
- \int_{\Gamma_{SD}} p^{(D)}(\mathbf{y}) \, u^{(*,D)}(\mathbf{x}, \mathbf{y}; \lambda) \, d\Gamma_{SD}(\mathbf{y}),
\end{align}
for all $\mathbf{x} \in \Omega_D$.

Recall that, analogously to the procedure described for $u^{(I)},$ in the previous expressions $u^{(*,j)}(\mathbf{x}, \mathbf{y}; \lambda)$ denotes the fundamental solution of the anisotropic Helmholtz operator associated with the tensor $\sigma^{(j)}$, now defined in the spatial region $\Omega_j$, and
\[
p^{(*,j)}(\mathbf{x}, \mathbf{y}; \lambda)
= \sigma^{(j)}(\mathbf{y}) \nabla_{\mathbf{y}} u^{(*,j)}(\mathbf{x}, \mathbf{y}; \lambda)\cdot \mathbf{n}(\mathbf{y}).
\]
for $j = S, D$.

The next step consists in determining the values of the electrical potential and the corresponding currents along the external boundary of $\Omega$ and at the interfaces between the subregions composing the domain. This is necessary to evaluate the single- and double-layer potentials on the right-hand side of the representation formulas \eqref{repreuieI}--\eqref{repreuieD}. These boundary values are obtained as the solution of a system of Fredholm integral equations of the second kind, leading to a coupled system of boundary integral equations. This system is derived by taking the limit in \eqref{repreuieI}--\eqref{repreuieD} as $\mathbf{x}$ approaches a point on the boundary.

Before proceeding further, we derive explicit expressions for the fundamental solutions associated with the electrical potentials, as well as for the corresponding current fluxes. In our setting, the governing operator is an anisotropic Helmholtz operator. As will be shown in the next section, these fundamental solutions can be expressed in terms of Bessel functions. 

\subsection{Fundamental solutions of Helmholtz equations for each region}

We begin by focusing on the region $\Omega_I$. We derive an explicit expression for the fundamental solution of the operator
\begin{align}\label{Helm1}
\Delta_{\mathbf{y}}^{\left(\sigma_l,\sigma_t\right)} + \lambda^2,
\end{align}
in terms of known functions. Here,
\begin{align*}
\Delta_{\mathbf{y}}^{\left(\sigma_l,\sigma_t\right)} = \sigma_l \frac{\partial^2}{\partial y_1^2} + \sigma_t \frac{\partial^2}{\partial y_2^2}
\end{align*}
denotes the anisotropic Laplacian with respect to the variable $\mathbf{y}$, which corresponds to an anisotropic Helmholtz operator with constant coefficients, and $\lambda$ is a real parameter.

As mentioned earlier, a distribution $u^{*} \in \mathcal{D}'(\mathbb{R}^2)$ (where $\mathcal{D}'(\mathbb{R}^2)$ denotes the space of distributions on $\mathbb{R}^2$, see Definition~1.1.4 in \cite{Ortner2015}) is said to be a fundamental solution of the operator $\Delta_{\mathbf{y}}^{\left(\sigma_l,\sigma_t\right)} + \lambda^2$ if it satisfies the equation
\begin{align*}
\Delta_{\mathbf{y}}^{\left(\sigma_l,\sigma_t\right)} u^{*} + \lambda^2 u^{*} = \delta \quad \text{in } \mathcal{D}'(\mathbb{R}^2),
\end{align*}
where $\delta$ denotes the Dirac delta distribution centered at the origin.

We now introduce the following change of variables:
\begin{align*}
\bar{y}_1 := \frac{y_1}{\sqrt{\sigma_l}}, \quad \bar{y}_2 := \frac{y_2}{\sqrt{\sigma_t}},
\end{align*}
which leads to the following equation in terms of the new variable $\bar{\mathbf{y}} = (\bar{y}_1, \bar{y}_2)$:
\begin{align}\label{Helm2}
\Delta_{\bar{\mathbf{y}}} u^{*} + \lambda u^{*} = \frac{1}{\sqrt{\sigma_l \sigma_t}} \, \delta \quad \text{in } \mathcal{D}'(\mathbb{R}^2_{\bar{\mathbf{y}}}),
\end{align}
where $\Delta$ denotes the classical Laplacian. We use the notation $\mathcal{D}'(\mathbb{R}^2_{\bar{\mathbf{y}}})$ to emphasize that equation \eqref{Helm2} is understood in the sense of distributions, with test functions defined in terms of the variable $\bar{\mathbf{y}}$.

On the left-hand side of \eqref{Helm2}, we now obtain the standard Helmholtz operator, while on the right-hand side the Dirac delta distribution is scaled by the factor $\dfrac{1}{\sqrt{\sigma_l \sigma_t}}$, as detailed in Example~1.2.8 of \cite{Ortner2015}. This transformation reduces the anisotropic problem to an equivalent isotropic one with constant coefficients.

The fundamental solution of the Helmholtz operator has been derived previously; for instance, \cite{Ortner2015} provides a general formula valid in arbitrary dimensions. In our case, we focus on the two-dimensional setting. Therefore, the fundamental solution to the Helmholtz problem \eqref{Helm1} is given by
\begin{align}
u^{(*,I)}(\mathbf{x}, \mathbf{y}; \lambda) = \frac{1}{4\sqrt{\sigma_l \sigma_t}} \, N_0\left(\lambda \, r^{(I)}(\mathbf{x}, \mathbf{y})\right),
\end{align}
where
\begin{equation*}
r^{(I)}(\mathbf{x}, \mathbf{y}) = \sqrt{ \frac{(x_1 - y_1)^2}{\sigma_l} + \frac{(x_2 - y_2)^2}{\sigma_t} },
\end{equation*}
which reflects the anisotropic scaling induced by the conductivity tensor, and $N_0(x)$ denotes the Bessel function of the second kind of order zero.

The expressions for the fundamental solutions in the regions $\Omega_S$ and $\Omega_D$ can be derived in an analogous manner. In these cases, we obtain
\begin{align}
u^{(*,S)}(\mathbf{x}, \mathbf{y}; \lambda) = \frac{1}{4\sqrt{\sigma_l \sigma_t}} \, N_0\left(\lambda \, r^{(S)}(\mathbf{x}, \mathbf{y})\right),
\end{align}
where
\begin{equation*}
r^{(S)}(\mathbf{x}, \mathbf{y}) = \sqrt{ \frac{(x_1 - y_1)^2}{\sigma_t} + \frac{(x_2 - y_2)^2}{\sigma_l} },
\end{equation*}
and
\begin{align}
u^{(*,D)}(\mathbf{x}, \mathbf{y}; \lambda) = \frac{1}{4\sqrt{\sigma_l \sigma_t}} \, N_0\left(\lambda \, r^{(D)}(\mathbf{x}, \mathbf{y})\right),
\end{align}
where
\begin{equation*}
r^{(D)}(\mathbf{x}, \mathbf{y}) = \sqrt{ \frac{(x_1 - y_1)^2}{\sigma_l} + \frac{(x_2 - y_2)^2}{\sigma_t} }.
\end{equation*}

An explicit expression for the Bessel function of the second kind of order zero, $N_0(x)$, is given by
\begin{align}
N_0(x) = \frac{2}{\pi} J_0(x) \left( \ln \frac{x}{2} + \gamma \right) 
- \frac{2}{\pi} \sum_{n = 0}^{\infty} (-1)^n \frac{1}{(n^2)!} \left( \frac{x}{2} \right)^{2n} \left( \sum_{j = 1}^{n} \frac{1}{j} \right),
\end{align}
where $\gamma$ denotes the Euler constant. It is important to note that $N_0(x)$ exhibits a logarithmic singularity at $x = 0$. This expansion can also be found in \cite{Bell1968}.

Furthermore, we derive an explicit expression for the function $p^{(*,I)}(\mathbf{x}, \mathbf{y}; \lambda)$. Given the structure of the matrix $\sigma^{(I)}$, we have
\begin{align*}
p^{(*,I)}(\mathbf{x}, \mathbf{y}; \lambda) = \nabla_{\mathbf{y}}^{\left(\sigma_l, \sigma_t\right)} u^{(*,I)}(\mathbf{x}, \mathbf{y}; \lambda) \cdot \mathbf{n}(\mathbf{y}),
\end{align*}
where the anisotropic gradient is defined as
\begin{align*}
\nabla_{\left(\sigma_l, \sigma_t\right)} = \left( \sigma_l \frac{\partial}{\partial y_1}, \; \sigma_t \frac{\partial}{\partial y_2} \right).
\end{align*}

Carrying out straightforward calculations, we obtain
\begin{align}
p^{(*,I)}(\mathbf{x}, \mathbf{y}; \lambda) 
= \frac{\lambda}{4\sqrt{\sigma_l \sigma_t}} 
\, \frac{(\mathbf{x} - \mathbf{y}) \cdot \mathbf{n}(\mathbf{y})}{r^{(I)}(\mathbf{x}, \mathbf{y})} 
\, N_1\left(\lambda r^{(I)}(\mathbf{x}, \mathbf{y})\right),
\end{align}
where $N_1(x)$ is the Bessel function of the second kind of order one, and $(\cdot, \cdot)$ denotes the standard scalar product in $\mathbb{R}^2$.

In deriving the above expression, we made use of the recurrence relation for the derivative of the Bessel function of the second kind, given by
\begin{align*}
\frac{d}{dx} N_n(x) = -N_{n+1}(x) + \frac{n}{x} N_n(x).
\end{align*}
In our case, this identity is applied with $n = 0$. For further details, see \cite{Korenev2002}.

The resulting expressions represent the fundamental current flux within the region $\Omega_I$.

In the other regions, we obtain
\begin{align}
p^{(*,S)}(\mathbf{x}, \mathbf{y}; \lambda) 
= \frac{\lambda}{4\sqrt{\sigma_l \sigma_t}} 
\, \frac{(\mathbf{x} - \mathbf{y}) \cdot \mathbf{n}(\mathbf{y})}{r^{(S)}(\mathbf{x}, \mathbf{y})} 
\, N_1\left(\lambda \, r^{(S)}(\mathbf{x}, \mathbf{y})\right),
\end{align}
and
\begin{align}
p^{(*,D)}(\mathbf{x}, \mathbf{y}; \lambda) 
= \frac{\lambda}{4\sqrt{\sigma_l \sigma_t}} 
\, \frac{(\mathbf{x} - \mathbf{y}) \cdot \mathbf{n}(\mathbf{y})}{r^{(D)}(\mathbf{x}, \mathbf{y})} 
\, N_1\left(\lambda \, r^{(D)}(\mathbf{x}, \mathbf{y})\right).
\end{align}

An explicit expression for $N_1(x)$, as provided in \cite{Korenev2002}, is given by
\begin{align*}
N_1(x) &= \frac{2}{\pi}J_1(x)\left(\ln \frac{x}{2} + \gamma\right) - \frac{1}{\pi}\left(\frac{x}{2}\right)^{-1}
- \frac{1}{\pi}\sum^{\infty}_{n = 0}(-1)^{n}\frac{\left(x/2\right)^{2n+1}}{n!(n+1)!}\left(\sum^{n}_{k = 1}\frac{1}{k} + \frac{1}{2(n+1)}\right).
\end{align*}

From the expression above, it follows that the function $N_1(x)$ exhibits a singularity at $x = 0$, with asymptotic behavior proportional to $1/x$. This singularity is stronger than that of the function $N_0(x)$. As a consequence, the single- and double-layer potentials, as expressed in \eqref{repreuieI}--\eqref{repreuieS}--\eqref{repreuieD}, may develop singular behavior as the point $\mathbf{x}$ approaches the boundary of $\Omega$ or the interfaces between the subregions composing $\Omega$.

In view of this, it is essential to analyze carefully the behavior of these potentials in the vicinity of the boundary. This will be the main focus of the following section.

\section{Boundary limits of single-layer and double-layer potentials}\label{BIEsect}
In order to formulate the boundary integral equations required to determine the boundary values of the intracellular and extracellular potentials and currents, it is essential to analyze the behavior of the single-layer and double-layer potentials near the boundary. In this section, we present the key results obtained for this purpose.

Before introducing the main findings, we shall establish some notation that will be used throughout the section.

We assume that \( \Omega \subset \mathbb{R}^2 \) is a bounded domain with a boundary \( \partial \Omega \) of piecewise \( C^2 \) regularity. Specifically, we assume the existence of a finite set of points \( \left\{\mathbf{z}_1,\ldots,\mathbf{z}_l\right\} \subset \partial\Omega \) such that, for every point \( \mathbf{z} \in \partial \Omega \setminus \left\{\mathbf{z}_1,\ldots,\mathbf{z}_l\right\} \), there exists a neighborhood \( V_\mathbf{z} \) of \( \mathbf{z} \) in which all points of \( V_\mathbf{z} \cap \partial \Omega \) can be locally represented by a parameterization of class \( C^2 \). Conversely, at the points \( \left\{\mathbf{z}_1,\ldots,\mathbf{z}_l\right\} \), the boundary \( \partial \Omega \) exhibits corners.

At each regular point \( \mathbf{z} \in \partial \Omega \setminus \left\{\mathbf{z}_1,\ldots,\mathbf{z}_l\right\} \), we define the unit tangent vector \( \mathbf{t}(\mathbf{z}) \) to \( \partial \Omega \), and choose its orientation such that the outward unit normal vector \( \mathbf{n}(\mathbf{z}) \) satisfies the standard right-hand rule. Note that both the tangent and normal vectors may exhibit discontinuities at the corner points \( \left\{\mathbf{z}_1,\ldots,\mathbf{z}_l\right\} \).

If \( D \) is a closed domain containing \( \partial \Omega \), we denote by \( C^{0,\alpha}(D) \), for \( 0 < \alpha < 1 \), the space of Hölder continuous functions on \( D \) with Hölder exponent \( \alpha \). We denote by \( C(\partial \Omega) \) the space of continuous functions defined on \( \partial \Omega \).

Now, let us introduce some notation. Let \( a, b \) be positive real constants. We define
\begin{align}\label{Phiab}
\Phi^{(a,b)}_{\lambda}(\mathbf{x}, \mathbf{y}) := \frac{1}{4\sqrt{ab}}N_0\left(\lambda\, r^{(a,b)}(\mathbf{x}, \mathbf{y})\right),
\end{align}
where
\begin{align*}
r^{(a,b)}(\mathbf{x}, \mathbf{y}) := \sqrt{\frac{(x_1 - y_1)^2}{a} + \frac{(x_2 - y_2)^2}{b}},
\end{align*}
and we assume that \( \lambda > 0 \). For \( \lambda = 0 \), we define
\begin{align*}
\Phi^{(a,b)}_{0}(\mathbf{x}, \mathbf{y}) := \frac{1}{\sqrt{ab}} \ln \left(r^{(a,b)}(\mathbf{x}, \mathbf{y})\right).
\end{align*}

As demonstrated earlier, we can establish that
\begin{align}\label{GradPhiab}
\nabla^{(a,b)}_{\mathbf{y}} \Phi_{\lambda}^{(a,b)}(\mathbf{x}, \mathbf{y}) \cdot \mathbf{n}(\mathbf{y}) 
= \frac{\lambda}{4\sqrt{ab}} \frac{(\mathbf{x} - \mathbf{y}, \mathbf{n}(\mathbf{y}))}{r^{(a,b)}(\mathbf{x}, \mathbf{y})} 
N_1\left(\lambda\, r^{(a,b)}(\mathbf{x}, \mathbf{y})\right),
\end{align}
where \( \nabla^{(a,b)}_{\mathbf{y}} \) denotes the anisotropic gradient with respect to the variable \( \mathbf{y} \), as previously defined.

In the case \( \lambda = 0 \), we have
\begin{align*}
\nabla^{(a,b)}_{\mathbf{y}} \Phi_{0}^{(a,b)}(\mathbf{x}, \mathbf{y}) \cdot \mathbf{n}(\mathbf{y}) 
= \frac{1}{\sqrt{ab}} \frac{\left(\mathbf{x} - \mathbf{y}, \mathbf{n}(\mathbf{y})\right)}{\left(r^{(a,b)}(\mathbf{x}, \mathbf{y})\right)^2}.
\end{align*}

Let us give the following proposition.
\begin{proposition}\label{prop1}
Let \( \mathbf{x}, \mathbf{y} \in \partial \Omega \), with \( \mathbf{y} \) belonging to a neighborhood of \( \mathbf{x} \). Then, there exists a constant \( L > 0 \) such that
\begin{align}
\left| \left( \mathbf{n}(\mathbf{y}), \mathbf{x} - \mathbf{y} \right) \right| &\leq L \left| \mathbf{x} - \mathbf{y} \right|^2, \label{inq0} \\
\left| \left( \mathbf{n}(\mathbf{x}), \mathbf{n}(\mathbf{y}) \right) \right| &\leq L \left| \mathbf{x} - \mathbf{y} \right|. \label{inq1}
\end{align}
\end{proposition}

\begin{proof}
In the case of \eqref{inq0}, assume that in a neighborhood \( V_{\mathbf{x}} \subset \mathbb{R}^2 \) of \( \mathbf{x} \), the boundary \( \partial \Omega \) admits a local \( C^2 \) parametrization. That is, for all \( \mathbf{y} \in V_{\mathbf{x}} \cap \partial \Omega \), we can write \( \mathbf{y} = \mathbf{y}(t) \) and \( \mathbf{x} = \mathbf{y}(t_0) \), with \( t \in I_{t_0} \), where \( I_{t_0} \) is a sufficiently small neighborhood of \( t_0 \).

Applying Taylor's expansion to \( \mathbf{y}(t) \) around \( t_0 \), we obtain
\begin{align*}
\left| \mathbf{x} - \mathbf{y}(t) \right|^2 = \left| \mathbf{y}(t_0) - \mathbf{y}(t) \right|^2 = \left| \mathbf{y}'(t_0^*) \right|^2 (t - t_0)^2,
\end{align*}
for some \( t_0^* \in I_{t_0} \), by the mean value theorem and the regularity of \( \mathbf{y} \). Hence, we deduce that
\begin{align*}
(t - t_0)^2 = \frac{1}{\left| \mathbf{y}'(t_0^*) \right|^2} \left| \mathbf{x} - \mathbf{y}(t) \right|^2.
\end{align*}
Since \( \left| \mathbf{y}'(t_0^*) \right| > 0 \), the inverse is bounded above by a constant in the neighborhood, and thus
\begin{align*}
(t - t_0)^2 \leq C \left| \mathbf{x} - \mathbf{y}(t) \right|^2
\end{align*}
for some constant \( C > 0 \). 

Next, let us consider the function
\[
f(t) := \mathbf{n}(\mathbf{y}(t)) \cdot \left( \mathbf{y}(t_0) - \mathbf{y}(t) \right),
\]
defined for \( t \in I_{t_0} \). Note that \( f(t_0) = 0 \) and \( f'(t_0) = 0 \). Indeed, we compute
\[
f(t_0) = \mathbf{n}(\mathbf{y}(t_0)) \cdot \left( \mathbf{y}(t_0) - \mathbf{y}(t_0) \right) = \mathbf{n}(\mathbf{y}(t_0)) \cdot \mathbf{0} = 0.
\]

Differentiating \( f(t) \) with respect to \( t \), we obtain
\[
f'(t) = \frac{d}{dt} \left( \mathbf{n}(\mathbf{y}(t)) \cdot \left( \mathbf{y}(t_0) - \mathbf{y}(t) \right) \right) = \left( \frac{d}{dt} \mathbf{n}(\mathbf{y}(t)) \right) \cdot \left( \mathbf{y}(t_0) - \mathbf{y}(t) \right) - \mathbf{n}(\mathbf{y}(t)) \cdot \mathbf{y}'(t).
\]
Evaluating at \( t = t_0 \), we find
\[
f'(t_0) = \left( \left. \frac{d}{dt} \mathbf{n}(\mathbf{y}(t)) \right|_{t = t_0} \right) \cdot \left( \mathbf{y}(t_0) - \mathbf{y}(t_0) \right) - \mathbf{n}(\mathbf{y}(t_0)) \cdot \mathbf{y}'(t_0) = 0.
\]

In the last term, we have used the fact that
\[
\mathbf{n}(\mathbf{y}(t_0)) \cdot \mathbf{y}'(t_0) = \mathbf{n}(\mathbf{x}) \cdot \mathbf{t}(\mathbf{x}) = 0,
\]
since the normal and tangent vectors at \( \mathbf{x} = \mathbf{y}(t_0) \) are orthogonal with respect to the Euclidean scalar product.

Thus, we obtain
\[
\mathbf{n}(\mathbf{y}) \cdot (\mathbf{x} - \mathbf{y}) = f''(t_0^{**})(t - t_0)^2 = 
\frac{f''(t_0^{**})}{2|\mathbf{y}'(t_0^{*})|}|\mathbf{x} - \mathbf{y}|^2, \quad \text{for all } t \in I_{t_0},
\]
for some \( t_0^{**} \in I_{t_0} \). Therefore, we conclude that there exists a constant \( L(\mathbf{x}) > 0 \), depending on \( \mathbf{x} \), such that
\[
\left| \mathbf{n}(\mathbf{y}) \cdot (\mathbf{x} - \mathbf{y}) \right| \leq L(\mathbf{x})|\mathbf{x} - \mathbf{y}|^2.
\]

Since \( \partial \Omega \) is compact, it is possible to show that there exists a constant \( L > 0 \), independent of \( \mathbf{x} \), for which the above inequality holds uniformly for all \( \mathbf{x}, \mathbf{y} \in \partial \Omega \).

The proof of inequality \eqref{inq1} follows by similar arguments.
\end{proof}

As discussed in previous sections, we focus on single-layer and double-layer potentials, defined as
\begin{align}
\mathcal{K}_{\lambda}^{(a,b)}[\varphi](\mathbf{x}) &:= \int_{\partial \Omega} \Phi_{\lambda}^{(a,b)}(\mathbf{x}, \mathbf{y})\, \varphi(\mathbf{y})\, ds(\mathbf{y}), \label{SP} \\
\mathcal{W}_{\lambda}^{(a,b)}[\varphi](\mathbf{x}) &:= \int_{\partial \Omega} \nabla^{(a,b)}_{\mathbf{y}} \Phi_{\lambda}^{(a,b)}(\mathbf{x}, \mathbf{y}) \cdot \mathbf{n}(\mathbf{y})\, \varphi(\mathbf{y})\, ds(\mathbf{y}), \label{DP}
\end{align}
where $\mathbf{x} \in \mathbb{R}^2 \setminus \partial \Omega$, and $\varphi$ is a function defined on $\partial \Omega$. Since $\mathbf{x}$ and $\mathbf{y}$ are distinct, it can be shown that both $\Phi_{\lambda}^{(a,b)}(\mathbf{x}, \mathbf{y})$ and $\nabla^{(a,b)}_{\mathbf{y}} \Phi_{\lambda}^{(a,b)}(\mathbf{x}, \mathbf{y}) \cdot \mathbf{n}(\mathbf{y})$ satisfy the anisotropic Helmholtz equation with respect to the variable $\mathbf{x}$. Therefore, $\mathcal{K}_{\lambda}^{(a,b)}[\varphi](\mathbf{x})$ and $\mathcal{W}_{\lambda}^{(a,b)}[\varphi](\mathbf{x})$ are also solutions in $\Omega$. Consequently, both the single-layer and double-layer potentials are analytic throughout $\Omega$.

However, it is important to note that these functions are not inherently defined on $\partial \Omega$ due to the presence of a singularity at $\mathbf{x} = \mathbf{y}$. Consequently, in the following sections, we will investigate the conditions under which the functions $\Phi_{\lambda}^{(a,b)}(\mathbf{x}, \mathbf{y})$ and $\nabla^{(a,b)}_{\mathbf{y}} \Phi_{\lambda}^{(a,b)}(\mathbf{x}, \mathbf{y}) \cdot \mathbf{n}(\mathbf{y})$ ensure that the corresponding potentials are well defined on $\partial \Omega$. We will show that it is sufficient for these functions to exhibit weak singularities.

The notion of a \emph{weakly singular kernel} is discussed in \cite{Colton2013}. However, we will introduce a slightly modified definition that is better suited to our specific context.

\begin{definition}
Let $K(\mathbf{x}, \mathbf{y})$ be a continuous function defined on $\partial \Omega \times \partial \Omega$, for $\mathbf{x} \neq \mathbf{y}$. We say that $K(\mathbf{x}, \mathbf{y})$ is \emph{weakly singular} on $\partial \Omega$ if, for every $\mathbf{x} \in \partial \Omega$, there exists a neighborhood $V_{\mathbf{x}} \subset \mathbb{R}^2$ and positive constants $L$ and $\alpha \in (0,2]$ such that
\begin{align}\label{condWS}
\left|K(\mathbf{x}, \mathbf{y})\right| \leq L \left|\mathbf{x} - \mathbf{y}\right|^{\alpha - 2},
\end{align}
for all $\mathbf{y} \in V_{\mathbf{x}} \cap \partial \Omega$ with $\mathbf{y} \neq \mathbf{x}$.
\end{definition}

\begin{theorem}\label{theo1}
The fundamental solution $\Phi_{\lambda}^{(a,b)}(\mathbf{x}, \mathbf{y})$ is weakly singular on $\partial \Omega$, whereas the expression $\nabla^{(a,b)}_{\mathbf{y}}\Phi_{\lambda}^{(a,b)}(\mathbf{x}, \mathbf{y})\cdot \mathbf{n}(\mathbf{y})$ remains bounded on $\partial \Omega \cap \left\{\mathbf{x} \neq \mathbf{y}\right\}$.
\end{theorem}
\begin{proof}
In the first place, observe that
\begin{align*}
\Phi_{\lambda}^{(a,b)}(\mathbf{x}, \mathbf{y}) = S_{\lambda}^{(a,b)}(\mathbf{x}, \mathbf{y}) + C_{\lambda}^{(a,b)}(\mathbf{x}, \mathbf{y}),
\end{align*}
where
\begin{align*}
S_{\lambda}^{(a,b)}(\mathbf{x}, \mathbf{y}) &= \frac{2}{\pi}J_0\left(\lambda r^{(a,b)}(\mathbf{x}, \mathbf{y})\right)\ln\left(\lambda \frac{r^{(a,b)}(\mathbf{x}, \mathbf{y})}{2}\right),\\
C^{(a,b)}(\mathbf{x}, \mathbf{y}) &= \frac{2\gamma}{\pi}J_0\left(\lambda r^{(a,b)}(\mathbf{x}, \mathbf{y})\right) - \frac{2}{\pi}\sum_{n = 0}^{\infty}\frac{(-1)^n}{(n^2)!}\left(\frac{\lambda r^{(a,b)}(\mathbf{x}, \mathbf{y})}{2}\right)^{2n}\sum_{j = 1}^{n}\frac{1}{j}.
\end{align*}
Here, \( S_{\lambda}^{(a,b)}(\mathbf{x}, \mathbf{y}) \) is the singular component of \( \Phi_{\lambda}^{(a,b)} \), and its weak singularity will be established below.

To that end, recall that
\[
\lim_{x \to 0^+} x\ln x = 0,
\]
which implies the existence of a constant \( \delta > 0 \) such that \( |\ln x| < x^{-1} \) for all \( 0 < x \leq \delta \). Consequently, for all \( \mathbf{x}, \mathbf{y} \in \partial \Omega \) with \( \mathbf{x} \neq \mathbf{y} \) and \( r^{(a,b)}(\mathbf{x}, \mathbf{y}) \leq 2\delta/\lambda \), we have
\begin{align*}
\left|\ln\left(\lambda \frac{r^{(a,b)}(\mathbf{x}, \mathbf{y})}{2}\right)\right| &\leq \frac{2}{\lambda r^{(a,b)}(\mathbf{x}, \mathbf{y})} \leq \frac{2}{\lambda} \max\left\{\frac{1}{\sqrt{a}}, \frac{1}{\sqrt{b}}\right\} \cdot \frac{1}{|\mathbf{x} - \mathbf{y}|}.
\end{align*}
This implies that \( \ln\left(\lambda \frac{r^{(a,b)}(\mathbf{x}, \mathbf{y})}{2}\right) \) is weakly singular on \( \partial \Omega \) with exponent \( \alpha = 1 \). Since \( J_0 \) is bounded on compact sets, the same property is inherited by \( S_{\lambda}^{(a,b)} \). Finally, the term \( C^{(a,b)} \) is a smooth function of \( r^{(a,b)} \), hence regular on \( \partial \Omega \). This proves the weak singularity of \( \Phi_{\lambda}^{(a,b)} \).

On the other hand, we know that $J_0(x)$ is a $C^{\infty}$ function in $\mathbb{R}$, and thus it is bounded if $\frac{\lambda}{2}r^{(a,b)}(\mathbf{x}, \mathbf{y}) \leq \delta$. We conclude that $S_{\lambda}^{(a,b)}(\mathbf{x}, \mathbf{y})$ is weakly singular on $\partial \Omega$ with $\alpha = 1$.

With respect to $\nabla^{(a,b)}_{\mathbf{y}}\Phi_{\lambda}^{(a,b)}(\mathbf{x}, \mathbf{y})\cdot \mathbf{n}(\mathbf{y})$, we have 
\begin{align*}
\nabla^{(a,b)}_{\mathbf{y}}\Phi_{\lambda}^{(a,b)}(\mathbf{x}, \mathbf{y})\cdot \mathbf{n}(\mathbf{y}) = B^{(a,b)}_1(\mathbf{x}, \mathbf{y}) + B^{(a,b)}_2(\mathbf{x}, \mathbf{y}) + B^{(a,b)}_3(\mathbf{x}, \mathbf{y}),
\end{align*}
where
\begin{align*}
B^{(a,b)}_1(\mathbf{x}, \mathbf{y}) = \frac{1}{2\sqrt{ab}\pi}\frac{\lambda\left(\mathbf{x} - \mathbf{y}, \mathbf{n}(\mathbf{y})\right)}{r^{(a,b)}(\mathbf{x}, \mathbf{y})}J_1(\lambda r^{(a,b)}(\mathbf{x}, \mathbf{y}))\ln \left(\frac{\lambda r^{(a,b)}(\mathbf{x}, \mathbf{y})}{2}\right),
\end{align*}
\begin{align*}
B^{(a,b)}_2(\mathbf{x}, \mathbf{y}) = -\frac{1}{2\sqrt{ab}\pi}\frac{\lambda\left(\mathbf{x} - \mathbf{y}, \mathbf{n}(\mathbf{y})\right)}{\left(r^{(a,b)}(\mathbf{x}, \mathbf{y})\right)^2},
\end{align*}
and 
\begin{align*}
B^{(a,b)}_3(\mathbf{x}, \mathbf{y}) = \frac{1}{2\sqrt{ab}\pi}\frac{\lambda\left(\mathbf{x} - \mathbf{y}, \mathbf{n}(\mathbf{y})\right)}{r^{(a,b)}(\mathbf{x}, \mathbf{y})}\left(\gamma J_1(\lambda r^{(a,b)}(\mathbf{x}, \mathbf{y})) - \frac{1}{2}C_1(\lambda r^{(a,b)}(\mathbf{x}, \mathbf{y}))\right),
\end{align*}
with 
\begin{align*}
C_1(x) = \sum^{\infty}_{n = 0}(-1)^{n}\frac{\left(x/2\right)^{2n+1}}{n!(n+1)!}\left(\sum^{n}_{k = 1}\frac{1}{k} + \frac{1}{2(n+1)}\right).
\end{align*}
Note that the function \( C_1 \) is the limit of continuous functions, and is therefore itself continuous. Consequently, \( C_1 \) is bounded on bounded sets.

Taking into account the previous considerations and Proposition~\ref{prop1}, we conclude that the function \( \nabla^{(a,b)}_{\mathbf{y}}\Phi^{(a,b)}(\mathbf{x}, \mathbf{y})\cdot \mathbf{n}(\mathbf{y}) \) is bounded on \( \partial \Omega \cap \left\{\mathbf{x} \neq \mathbf{y}\right\} \). 
\end{proof}

Let us state the following theorem, which enables us to define the potentials \eqref{SP} and \eqref{DP} on the boundary $\partial \Omega$.

\begin{theorem}\label{theo2}
Let $K(\mathbf{x}, \mathbf{y})$ be weakly singular on $\partial \Omega$. Then, for all $\varphi \in C(\partial \Omega)$, the integral
\begin{align*}
\mathcal{A}(\varphi)(\mathbf{x}) := \int_{\partial \Omega} K(\mathbf{x}, \mathbf{y}) \varphi(\mathbf{y})\, ds(\mathbf{y}),
\end{align*}
is finite as an improper integral for each $\mathbf{x} \in \partial \Omega$. Moreover, $\mathcal{A}(\varphi)(\mathbf{x})$ defines a continuous function on $\partial \Omega$, and the operator $\mathcal{A}: C(\partial \Omega) \rightarrow C(\partial \Omega)$ is compact.
\end{theorem}
\begin{proof}
Let $\mathbf{x} \in \partial \Omega$, and let $V_{\mathbf{x}}$ be a neighborhood of $\mathbf{x}$ such that the condition \eqref{condWS} holds for all $\mathbf{y} \in V_{\mathbf{x}} \cap \partial \Omega$. Furthermore, let $R > 0$ be a constant such that the set
\[
S_{\mathbf{x},R} := \left\{\mathbf{y} \in \partial \Omega \mid \left|\mathbf{x} - \mathbf{y}\right| < R\right\}
\]
is contained in $V_{\mathbf{x}} \cap \partial \Omega$.

We can rewrite the integral $\mathcal{A}(\varphi)(\mathbf{x})$ as
\[
\mathcal{A}(\varphi)(\mathbf{x})
  = \int_{\partial \Omega \setminus S_{\mathbf{x},R}} K(\mathbf{x}, \mathbf{y})\, \varphi(\mathbf{y})\, ds(\mathbf{y})
  + \int_{S_{\mathbf{x},R}} K(\mathbf{x}, \mathbf{y})\, \varphi(\mathbf{y})\, ds(\mathbf{y}).
\]
First, since 
\(
(\partial \Omega \setminus S_{\mathbf{x},R}) \times (\partial \Omega \setminus S_{\mathbf{x},R})
\) 
is compact and $K(\mathbf{x},\mathbf{y})$ is continuous on this set, it is bounded. Hence
\[
\left|\int_{\partial \Omega \setminus S_{\mathbf{x},R}} K(\mathbf{x}, \mathbf{y})\, \varphi(\mathbf{y})\, ds(\mathbf{y})\right| < \infty.
\]

Now, let us estimate the integral over \( S_{\mathbf{x},R} \). For this purpose, define \( \mathbf{y}_r \) as the orthogonal projection of \( \mathbf{y} \in S_{\mathbf{x},R} \) onto the tangent line to $S_{\mathbf{x},R}$ at the point $\mathbf{x}$.

Again, we have 
\begin{align*}
\left|\int_{S_{\mathbf{x},R}}K(\mathbf{x}, \mathbf{y})\varphi(\mathbf{y})ds(\mathbf{y})\right| \leq \|\varphi\|_{\infty}\int_{S_{\mathbf{x},R}}\left|\mathbf{x} - \mathbf{y}\right|^{\alpha-2}ds(\mathbf{y}).
\end{align*}
On the other hand, if \( \mathbf{y}_r \) denotes the aforementioned projection, then
\begin{align*}
\mathbf{y}_r - \mathbf{x} = \kappa\, \mathbf{t}(\mathbf{x}),
\end{align*}
where \( \mathbf{t}(\mathbf{x}) \) is the unit tangent vector to \( \partial \Omega \) at \( \mathbf{x} \), and \( \kappa \in \mathbb{R} \) is the scalar parameter that determines the position of \( \mathbf{y}_r \) along the tangent line. In the specific case of the projection \( \mathbf{y}_r \), this parameter is given by
\begin{align*}
\kappa = \frac{(\mathbf{x} - \mathbf{y}) \cdot \mathbf{n}(\mathbf{y})}{1 - \mathbf{n}(\mathbf{y}) \cdot \mathbf{t}(\mathbf{x})}.
\end{align*}
Due to the above, we obtain
\begin{align*}
\left|\mathbf{y}_r - \mathbf{x}\right| 
&= |\kappa| 
= \left| \frac{(\mathbf{x} - \mathbf{y}) \cdot \mathbf{n}(\mathbf{y})}{1 - \mathbf{n}(\mathbf{y}) \cdot \mathbf{t}(\mathbf{x})} \right| \\
&\leq \frac{\left|(\mathbf{x} - \mathbf{y}) \cdot \mathbf{n}(\mathbf{y})\right|}{1 - \mathbf{n}(\mathbf{y}) \cdot \mathbf{t}(\mathbf{x})}
\leq \frac{L\,\left|\mathbf{x} - \mathbf{y}\right|^2}{1 - \mathbf{n}(\mathbf{y}) \cdot \mathbf{t}(\mathbf{x})},
\end{align*}
where we used Proposition~\ref{prop1} to bound the numerator. We assume that there exists a constant \( 0 < l < 1 \) such that \( \mathbf{n}(\mathbf{y}) \cdot \mathbf{t}(\mathbf{x}) < l \) for all \( \mathbf{y} \in S_{\mathbf{x}, R} \). Hence, we finally obtain
\begin{align*}
\left|\mathbf{y}_r - \mathbf{x}\right| \leq \frac{L}{1 - l} \left|\mathbf{x} - \mathbf{y}\right|^2.
\end{align*}

Finally, we have obtained that 
\begin{align*}
\left|\mathbf{x} - \mathbf{y}\right| \geq C \left|\mathbf{x} - \mathbf{y}_r\right|^{1/2},
\end{align*}
thus,
\begin{align*}
\left|\int_{S_{\mathbf{x},R}}K(\mathbf{x}, \mathbf{y})\varphi(\mathbf{y})ds(\mathbf{y})\right| &\leq \|\varphi\|_{\infty}\int_{S_{\mathbf{x},R}}\left|\mathbf{x} - \mathbf{y}\right|^{\alpha-2}ds(\mathbf{y})
\leq \|\varphi\|_{\infty}\int_{S_{\mathbf{x},R}}\left|\mathbf{x} - \mathbf{y}_r\right|^{\alpha/2 - 1}ds(\mathbf{y}_r) \\
&\leq  2\|\varphi\|_{\infty}\int^R_{0}\kappa^{\alpha/2 - 1}d\kappa = \frac{4\|\varphi\|_{\infty}}{\alpha} R^{\alpha/2} < \infty.
\end{align*}
\end{proof}
\begin{remark}
As a direct consequence of Theorems~\ref{theo1} and~\ref{theo2}, we conclude that the single- and double-layer potentials, denoted by $\mathcal{K}_{\lambda}^{(a,b)}[\varphi](\mathbf{x})$ and $\mathcal{W}_{\lambda}^{(a,b)}[\varphi](\mathbf{x})$ respectively, as defined in formulas~\eqref{SP} and~\eqref{DP}, are well-defined for every point $\mathbf{x} \in \partial \Omega$, provided the integrals are understood in the improper sense. Moreover, the associated operators
\begin{align*}
&\mathcal{K}_{\lambda}^{(a,b)}: C(\partial \Omega) \rightarrow C(\partial \Omega), \\
&\mathcal{W}_{\lambda}^{(a,b)}: C(\partial \Omega) \rightarrow C(\partial \Omega),
\end{align*}
are compact.
\end{remark}
\begin{theorem}\label{theo3}
Let $D$ be a closed domain containing $\partial \Omega$. Consider $K(\mathbf{x}, \mathbf{y})$ a function defined and continuous for all $\mathbf{x} \in D$ and $\mathbf{y} \in \partial \Omega$, with $\mathbf{x} \neq \mathbf{y}$. Assume that $K$ satisfies the following conditions:
\begin{enumerate}
\item For every $\mathbf{x} \in D$, there exists a neighborhood $V_{\mathbf{x}} \subset D$, a constant $L > 0$ independent of $\mathbf{x}$, and a real number $\alpha \in (0,2]$ such that
\begin{align*}
\left|K(\mathbf{x},\mathbf{y})\right| \leq L \left|\mathbf{x} - \mathbf{y}\right|^{\alpha - 2}, \quad \text{for all } \mathbf{y} \in V_{\mathbf{x}} \cap \partial \Omega.
\end{align*}

\item There exists an integer $m \in \mathbb{N}$ such that
\begin{align*}
\left|K(\mathbf{x}_1,\mathbf{y}) - K(\mathbf{x}_2,\mathbf{y})\right| \leq L \sum_{j = 1}^{m} \left|\mathbf{x}_1 - \mathbf{y}\right|^{\alpha - 2 - j} \left|\mathbf{x}_1 - \mathbf{x}_2\right|^{j},
\end{align*}
for all $\mathbf{x}_1, \mathbf{x}_2 \in V_{\mathbf{x}}$ and $\mathbf{y} \in V_{\mathbf{x}} \cap \partial \Omega$, provided that $2\left|\mathbf{x}_1 - \mathbf{x}_2\right| \leq \left|\mathbf{x}_1 - \mathbf{y}\right|$.
\end{enumerate}

Then, for every $\varphi \in C(\partial \Omega)$, the potential
\begin{align*}
\mathcal{A}(\varphi)(\mathbf{x}) = \int_{\partial \Omega}K(\mathbf{x}, \mathbf{y})\varphi(\mathbf{y})ds(\mathbf{y}),\quad \mathbf{x} \in D,
\end{align*}
belongs to the Hölder space $C^{0,\beta}(D)$ for every
\begin{itemize}
\item $\beta \in (0, \alpha]$ if $0 < \alpha < 1$,
\item $\beta \in (0, 1)$ if $\alpha = 1$,
\item $\beta \in (0, 1]$ if $1 < \alpha < 2$.
\end{itemize}
Moreover, the following estimate holds:
\begin{align*}
\left\|\mathcal{A}(\varphi)\right\|_{C^{0,\beta}(D)} \leq C_{\beta}\left\|\varphi\right\|_{C(\partial \Omega)},
\end{align*}
where $C_\beta > 0$ is a constant depending only on $\beta$.
\end{theorem}
The proof of this theorem follows from Theorem 2.7 in \cite{Colton2013}, with appropriate modifications to account for the fact that we are working in two spatial dimensions, and for the local nature of the definition of weakly singular kernels adopted in our context.

We now proceed to examine how the theorem above applies to the specific form of the fundamental solution considered in this work.
\begin{proposition}
Consider $S_{\lambda}^{(a,b)}(\mathbf{x}, \mathbf{y})$, the weakly singular component of the fundamental solution $\Phi_{\lambda}^{(a,b)}(\mathbf{x}, \mathbf{y})$. There exists a closed domain $D \subset \mathbb{R}^2$ containing $\partial \Omega$ such that $S_{\lambda}^{(a,b)}(\mathbf{x}, \mathbf{y})$ satisfies the conditions of Theorem~\ref{theo3} within $D$.
\end{proposition}

\begin{proof}
Let $\mathbf{z} \in \partial \Omega$ and let $V_{\mathbf{z}} \subset \mathbb{R}^2$ be a neighborhood of $\mathbf{z}$. It is possible to choose $V_{\mathbf{z}}$ such that the following conditions hold for all $\mathbf{x}_1, \mathbf{x}_2 \in V_{\mathbf{z}}$ and $\mathbf{y} \in V_{\mathbf{z}} \cup \partial \Omega$:
\begin{align*}
&\left|(\mathbf{x}_1 - \mathbf{y})(\mathbf{x}_2 - \mathbf{y})\ln \left(\frac{\lambda}{2}r^{(a,b)}(\mathbf{x}_1, \mathbf{y})\right)\right| \leq M_{\mathbf{z}},\\
&\left|(\mathbf{x}_1 - \mathbf{y})(\mathbf{x}_2 - \mathbf{y})\ln \left(\frac{\lambda}{2}r^{(a,b)}(\mathbf{x}_2, \mathbf{y})\right)\right| \leq M_{\mathbf{z}},\\
&\ln \left(\dfrac{\lambda}{2}r^{(a,b)}(\mathbf{x}_1, \mathbf{y})\right) < 0,\quad 
\ln \left(\dfrac{\lambda}{2}r^{(a,b)}(\mathbf{x}_2, \mathbf{y})\right) < 0,\\
&J_0\left(\lambda r^{(a,b)}(\mathbf{x}_1, \mathbf{y})\right)J_0\left(\lambda r^{(a,b)}(\mathbf{x}_2, \mathbf{y})\right) > 0,\\
&\left|J_0\left(\lambda r^{(a,b)}(\mathbf{x}_1, \mathbf{y})\right) - J_0\left(\lambda r^{(a,b)}(\mathbf{x}_2, \mathbf{y})\right)\right| \leq L_{\mathbf{z}}\left|\mathbf{x}_1 - \mathbf{x}_2\right|.
\end{align*}

Under these assumptions, we estimate the difference:
{\footnotesize
\begin{align*}
&\left|S_{\lambda}^{(a,b)}(\mathbf{x}_1, \mathbf{y}) - S_{\lambda}^{(a,b)}(\mathbf{x}_2, \mathbf{y})\right| \\
&\leq \frac{\big|\left|(\mathbf{x}_1 - \mathbf{y})(\mathbf{x}_2 - \mathbf{y})\ln \left(\frac{\lambda}{2}r^{(a,b)}(\mathbf{x}_1, \mathbf{y})\right)J_0\left(\lambda r^{(a,b)}(\mathbf{x}_1, \mathbf{y})\right)\right| - \left|(\mathbf{x}_1 - \mathbf{y})(\mathbf{x}_2 - \mathbf{y})\ln \left(\frac{\lambda}{2}r^{(a,b)}(\mathbf{x}_2, \mathbf{y})\right)J_0\left(\lambda r^{(a,b)}(\mathbf{x}_2, \mathbf{y})\right)\right|\big|}{\left|\mathbf{x}_1 - \mathbf{y}\right|^2} \\
&\leq \frac{M_{\mathbf{z}}\left|\left|J_0\left(\lambda r^{(a,b)}(\mathbf{x}_1, \mathbf{y})\right)\right| - \left|J_0\left(\lambda r^{(a,b)}(\mathbf{x}_2, \mathbf{y})\right)\right|\right|}{\left|\mathbf{x}_1 - \mathbf{y}\right|^2} \leq M_{\mathbf{z}}L_{\mathbf{z}}\left|\mathbf{x}_1 - \mathbf{y}\right|^{-2}\left|\mathbf{x}_1 - \mathbf{x}_2\right|.
\end{align*}
}
Therefore, $S_{\lambda}^{(a,b)}(\mathbf{x}, \mathbf{y})$ satisfies condition~(2) of Theorem~\ref{theo3} in $V_{\mathbf{z}}$ with $m = 1$ and $\alpha = 1$, for every $\mathbf{z} \in \partial \Omega$. Since $\partial \Omega$ is compact, the constants $M_{\mathbf{z}}$ and $L_{\mathbf{z}}$ can be chosen uniformly across all $\mathbf{z}$.

Moreover, by choosing $V_{\mathbf{z}}$ sufficiently small, it is straightforward to verify that condition~(1) of Theorem~\ref{theo3} is also satisfied within $V_{\mathbf{z}}$.

Finally, the domain $D$ can be taken as the union $\bigcup_{\mathbf{z} \in \partial \Omega} V_{\mathbf{z}}$, which is closed and contains $\partial \Omega$.
\end{proof}

\begin{remark}\label{rem1}
Based on the preceding proposition, we may apply Theorem~\ref{theo3} to the weakly singular part of the fundamental solution $\Phi_{\lambda}^{(a,b)}(\mathbf{x}, \mathbf{y})$. As a result, there exists a domain $D \subset \mathbb{R}^2$ containing $\partial \Omega$ such that the operator
\begin{align*}
\mathcal{K}^{(a,b)}_s: \varphi \in C(\partial \Omega) \mapsto \int_{\partial \Omega} S^{(a,b)}_{\lambda}(\mathbf{x}, \mathbf{y})\,\varphi(\mathbf{y})\,ds(\mathbf{y}), \quad \mathbf{x} \in D,
\end{align*}
acts continuously from $C(\partial \Omega)$ into the Hölder space $C^{0,\beta}(D)$ for some $\beta \in (0,1]$.

In contrast, for the continuous part $C^{(a,b)}(\mathbf{x}, \mathbf{y})$ of the fundamental solution, the operator
\begin{align*}
\mathcal{K}^{(a,b)}_c(\varphi)(\mathbf{x}) := \int_{\partial \Omega} C^{(a,b)}(\mathbf{x}, \mathbf{y})\,\varphi(\mathbf{y})\,ds(\mathbf{y}), \quad \mathbf{x} \in D,
\end{align*}
is well-defined and yields a function in $C^0(D)$ for every $\varphi \in C(\partial \Omega)$.
\end{remark}
Remark~\ref{rem1} holds practical significance for the problem under consideration. Specifically, it implies that
\begin{align}\label{lim1}
\lim_{\mathbf{x} \to \mathbf{z}} \mathcal{K}_{\lambda}^{(a,b)}[\varphi](\mathbf{x}) = \mathcal{K}_{\lambda}^{(a,b)}[\varphi](\mathbf{z}),
\end{align}
where the limit is taken as $\mathbf{x}$ approaches $\mathbf{z}$ from within either $\Omega$ or $\overline{\Omega}^c$, for every $\mathbf{z} \in \partial \Omega$ and $\varphi \in C(\partial \Omega)$. This continuity property stands in contrast to the behavior of the double-layer potential $\mathcal{W}_{\lambda}^{(a,b)}[\varphi](\cdot)$, which exhibits a removable discontinuity on $\partial \Omega$ for any $\varphi \in C(\partial \Omega)$.

We now aim to determine the limiting values described in \eqref{lim1} for the double-layer potential $\mathcal{W}_{\lambda}^{(a,b)}[\varphi](\mathbf{x})$, where $\varphi \in C(\partial \Omega)$ and $\mathbf{x} \in \mathbb{R}^2 \setminus \partial \Omega$. In this case, the kernel associated with the double-layer potential, $\nabla^{(a,b)}_{\mathbf{y}}\Phi^{(a,b)}_{\lambda}(\mathbf{x}, \mathbf{y})\cdot\mathbf{n}(\mathbf{y})$, does not satisfy the assumptions of Theorem~\ref{theo3}. Consequently, continuity of the potential on $\partial \Omega$ cannot be guaranteed. 

In fact, if we denote
\begin{align*}
v^{\pm}(\mathbf{z}) := \lim_{\mathbf{x} \to \mathbf{z}} \mathcal{W}_{\lambda}^{(a,b)}[\varphi](\mathbf{x}), \quad \text{with $\mathbf{x} \in \Omega$ or $\mathbf{x} \in \overline{\Omega}^c,$ respectively,}
\end{align*}
then it follows that
\begin{align}\label{plejform}
v^{\pm}(\mathbf{z}) = \mathcal{W}_{\lambda}^{(a,b)}[\varphi](\mathbf{z}) \mp \frac{1}{2}\varphi(\mathbf{z}), \quad \text{for all $\mathbf{z} \in \partial \Omega$.}
\end{align}
Note that, in particular, this implies that $\mathcal{W}_{\lambda}^{(a,b)}[\varphi](\mathbf{x})$ exhibits a jump discontinuity across $\partial \Omega$, with the magnitude of the jump being precisely $\varphi(\mathbf{z})$. We provide only a sketch of the proof, as it is relatively straightforward. 

To this end, observe that the difference 
\[
\nabla_{\mathbf{y}}^{(a,b)}\Phi_{\lambda}^{(a,b)}(\mathbf{x}, \mathbf{y}) \cdot \mathbf{n}(\mathbf{y}) - \nabla_{\mathbf{y}}^{(a,b)}\Phi_{0}^{(a,b)}(\mathbf{x}, \mathbf{y}) \cdot \mathbf{n}(\mathbf{y})
\]
satisfies the conditions of Theorem~\ref{theo3}. Therefore, the associated double-layer potential is continuous on $\partial \Omega$. It follows that the jump discontinuity arises entirely from the term
\begin{align*}
\mathcal{W}^{(a,b)}_0(\varphi)(\mathbf{x}) := \int_{\partial \Omega}\nabla_{\mathbf{y}}^{(a,b)}\Phi_{0}^{(a,b)}(\mathbf{x}, \mathbf{y})\cdot\mathbf{n}(\mathbf{y})\varphi(\mathbf{y})\,ds(\mathbf{y}),
\end{align*}
which corresponds to the double-layer potential for the anisotropic Laplace operator. This kernel has been thoroughly analyzed in the literature. For example, Brebbia et al.~\cite{Brebbia1984} show that the potential $\mathcal{W}^{(a,b)}_0(\varphi)(\mathbf{x})$ satisfies the jump relations stated in \eqref{plejform}. 

We will now summarize the analysis presented above in the following proposition.
\begin{proposition}\label{propLim}
Let $\Omega \subset \mathbb{R}^2$ be a bounded domain with piecewise smooth boundary $\partial \Omega$ of class $C^2$, and let $\left\{\mathbf{z}_1,\ldots,\mathbf{z}_l\right\}$ denote the corner points. Let $\Phi_{\lambda}^{(a,b)}(\mathbf{x}, \mathbf{y})$ be the fundamental solution of the anisotropic Helmholtz operator, and let $\nabla_{\mathbf{y}}^{(a,b)}\Phi_{\lambda}^{(a,b)}(\mathbf{x}, \mathbf{y}) \cdot \mathbf{n}(\mathbf{y})$ denote its anisotropic normal derivative, both defined on $\mathbb{R}^2 \setminus \partial \Omega$. 

For any $\mathbf{z} \in \partial \Omega$ and $\varphi \in C(\partial \Omega)$, define the non-tangential limits:
\begin{align*}
u_{\lambda}^{\pm}(\mathbf{z}) &= \lim_{\mathbf{x} \to \mathbf{z}} \mathcal{K}_{\lambda}^{(a,b)}[\varphi](\mathbf{x}), \\
v_{\lambda}^{\pm}(\mathbf{z}) &= \lim_{\mathbf{x} \to \mathbf{z}} \mathcal{W}_{\lambda}^{(a,b)}[\varphi](\mathbf{x}),
\end{align*}
where the limits are taken as $\mathbf{x}$ approaches $\mathbf{z}$ from within $\Omega$ and from within $\overline{\Omega}^c$, respectively. 

If the kernel $\Phi_{\lambda}^{(a,b)}(\mathbf{x}, \mathbf{y})$ satisfies the conditions of Theorems~\ref{theo2} and~\ref{theo3}, and the kernel $\nabla_{\mathbf{y}}^{(a,b)}\Phi_{\lambda}^{(a,b)}(\mathbf{x}, \mathbf{y}) \cdot \mathbf{n}(\mathbf{y})$ satisfies the conditions of Theorem~\ref{theo3}, then the following relations hold:
\begin{align}
\nonumber
u_{\lambda}^{+}(\mathbf{z}) &= u_{\lambda}^{-}(\mathbf{z}) = \mathcal{K}_{\lambda}^{(a,b)}[\varphi](\mathbf{z}), \\
v_{\lambda}^{\pm}(\mathbf{z}) &= \mathcal{W}_{\lambda}^{(a,b)}[\varphi](\mathbf{z}) + c(\mathbf{z})\varphi(\mathbf{z}). \label{forlimit1}
\end{align}
Here, the coefficient $c(\mathbf{z})$ is given by
\begin{align*}
c(\mathbf{z}) = 
\begin{cases}
\mp \dfrac{1}{2}, & \text{if } \mathbf{z} \in \partial \Omega \setminus \left\{\mathbf{z}_1,\ldots,\mathbf{z}_l\right\}, \\
\mp \dfrac{\theta(\mathbf{z})}{2\pi}, & \text{if } \mathbf{z} \in \left\{\mathbf{z}_1,\ldots,\mathbf{z}_l\right\},
\end{cases}
\end{align*}
where $\theta(\mathbf{z})$ denotes the internal angle at the corner point $\mathbf{z}$. 

The singular integral appearing on the right-hand side of \eqref{forlimit1} is understood in the improper sense.
\end{proposition}
Finally, we present a result concerning the anisotropic normal derivatives of the limiting values of the single- and double-layer potentials.
\begin{proposition}
Let $\varphi \in C(\partial \Omega)$. Then, the limits of the anisotropic normal derivatives of the single- and double-layer potentials satisfy the following:
\begin{align}
\nabla^{(a,b)}_{\mathbf{x}}\left(\mathcal{K}^{(a,b)}_{\lambda}[\varphi](\mathbf{z})\right)^{\pm} \cdot \mathbf{n}(\mathbf{z}) &= \int_{\partial \Omega}\nabla^{(a,b)}_{\mathbf{x}}\Phi^{(a,b)}_{\lambda}(\mathbf{z}, \mathbf{y}) \cdot \mathbf{n}(\mathbf{y})\, \varphi(\mathbf{y})\, ds(\mathbf{y}) - c(\mathbf{z})\varphi(\mathbf{z}), \label{limit3} \\[1.5ex]
\nabla^{(a,b)}_{\mathbf{x}}\left(\mathcal{W}^{(a,b)}_{\lambda}[\varphi](\mathbf{z})\right)^{+} \cdot \mathbf{n}(\mathbf{z}) &= \nabla^{(a,b)}_{\mathbf{x}}\left(\mathcal{W}^{(a,b)}_{\lambda}[\varphi](\mathbf{z})\right)^{-} \cdot \mathbf{n}(\mathbf{z}). \label{limit4}
\end{align}
Here, the limits are taken as
\begin{align*}
\nabla^{(a,b)}_{\mathbf{x}}\left(\mathcal{K}^{(a,b)}_{\lambda}[\varphi](\mathbf{z})\right)^{\pm} \cdot \mathbf{n}(\mathbf{z}) &= \lim_{h \to 0^{\pm}} \nabla^{(a,b)}_{\mathbf{x}}\left(\mathcal{K}^{(a,b)}_{\lambda}[\varphi](\mathbf{z} + h \mathbf{n}(\mathbf{z}))\right) \cdot \mathbf{n}(\mathbf{z}), \\
\nabla^{(a,b)}_{\mathbf{x}}\left(\mathcal{W}^{(a,b)}_{\lambda}[\varphi](\mathbf{z})\right)^{\pm} \cdot \mathbf{n}(\mathbf{z}) &= \lim_{h \to 0^{\pm}} \nabla^{(a,b)}_{\mathbf{x}}\left(\mathcal{W}^{(a,b)}_{\lambda}[\varphi](\mathbf{z} + h \mathbf{n}(\mathbf{z}))\right) \cdot \mathbf{n}(\mathbf{z}).
\end{align*}
\end{proposition}
\begin{proof}
Let $\mathbf{x} = \mathbf{z} + h\, \mathbf{n}(\mathbf{z})$. Then, we can write:
\begin{align*}
\nabla^{(a,b)}_{\mathbf{x}}\left(\mathcal{K}^{(a,b)}_{\lambda}[\varphi](\mathbf{x})\right)\cdot\mathbf{n}(\mathbf{z}) + \mathcal{W}^{(a,b)}_{\lambda}[\varphi](\mathbf{x}) = \int_{\partial \Omega}\nabla^{(a,b)}_{\mathbf{y}}\Phi^{(a,b)}_{\lambda}(\mathbf{x}, \mathbf{y}) \cdot \left(\mathbf{n}(\mathbf{y}) - \mathbf{n}(\mathbf{z})\right)\varphi(\mathbf{y})\, ds(\mathbf{y}).
\end{align*}

It can be shown that the kernel on the right-hand side satisfies the conditions of Theorem~\ref{theo3}. Indeed, since $\mathbf{n}(\mathbf{y})$ is smooth and $\mathbf{n}(\mathbf{y}) - \mathbf{n}(\mathbf{z})$ vanishes as $\mathbf{y} \to \mathbf{z}$, the integrand becomes weakly singular with sufficient decay. Consequently, the integral defines a continuous function on $\partial \Omega$, and its limit as $\mathbf{x} \to \mathbf{z}$ exists and vanishes.

This establishes the validity of formula~\eqref{limit3}. As for the identity in~\eqref{limit4}, we omit its proof here, as it follows from similar estimates to those already used in previous results.
\end{proof}

\section{Obtaining the System of Singular Boundary Integral Equations (BIEs) Associated with the Boundary Problem}

We begin by briefly summarizing the results obtained so far. First, we formulated the boundary value problem that we aim to solve. This is an eigenvalue problem for the electrical potential $u(\mathbf{x})$, defined for $\mathbf{x} \in \Omega$, subject to homogeneous boundary conditions on $\partial_{\text{ext}}\Omega$ for the current flux, together with contact conditions on $\Gamma_{IS}$ and $\Gamma_{SD}$. The complete formulation of this boundary value problem is given by equations~\eqref{Eqn1}, \eqref{BC1}, \eqref{PCC1}, and \eqref{PCC2}.

In Section~\ref{DirForSec}, we showed that the solutions can be represented as the sum of a single-layer potential and a double-layer potential. These representations are provided in Equations~\eqref{repreuieI}, \eqref{repreuieS}, and \eqref{repreuieD}, corresponding to $u^{(I)}_{i,e}$, $u^{(S)}_{i,e}$, and $u^{(D)}_{i,e}$, respectively. We refer to this representation as the \emph{direct formulation}.

The direct formulation reduces the solution of the original boundary value problem to the determination of the boundary values $u^{(I)}$ and $p^{(I)}$ on $\partial \Omega_I$, $u^{(S)}$ and $p^{(S)}$ on $\partial \Omega_S$, and $u^{(D)}$ and $p^{(D)}$ on $\partial \Omega_D$. These quantities will be treated as the unknowns from this point forward.

However, the boundary and contact conditions alone are not sufficient to fully determine these unknowns. To overcome this, we employ the limiting expressions provided in Proposition~\ref{propLim} to derive a system of singular boundary integral equations (BIEs) governing the unknown boundary values.

With the aim of applying the results from Section~\ref{BIEsect}, we observe that the fundamental solutions associated with the boundary value problems for the electrical potential can be written as
\begin{align*}
u^{(*,I)}(\mathbf{x}, \mathbf{y}; \lambda) &= \Phi^{\left(\sigma_l, \sigma_l\right)}_{\lambda}(\mathbf{x}, \mathbf{y}),\\[0.2cm]
p^{(*,I)}(\mathbf{x}, \mathbf{y}; \lambda) &= \nabla^{\left(\sigma_l,\sigma_t\right)}_{\mathbf{y}} \Phi_{\lambda}^{\left(\sigma_l,\sigma_t\right)}(\mathbf{x}, \mathbf{y}) \cdot \mathbf{n}(\mathbf{y}),
\end{align*}
where $\lambda \geq 0$, and $\mathbf{x}, \mathbf{y} \in \Omega_I$ with $\mathbf{x} \neq \mathbf{y}$. The right-hand sides correspond to the expressions given in \eqref{Phiab} and~\eqref{GradPhiab}, with $a = \sigma^{(i,e)}_l$ and $b = \sigma^{(i,e)}_t$. 

Similarly, for the solution in $\Omega_S$, we obtain
\begin{align*}
u^{(*,S)}(\mathbf{x}, \mathbf{y}; \lambda) &= \Phi^{\left(\sigma_t, \sigma_l\right)}_{\lambda}(\mathbf{x}, \mathbf{y}),\\[0.2cm]
p^{(*,S)}(\mathbf{x}, \mathbf{y}; \lambda) &= \nabla^{\left(\sigma_t,\sigma_l\right)}_{\mathbf{y}} \Phi_{\lambda}^{\left(\sigma_t,\sigma_l\right)}(\mathbf{x}, \mathbf{y}) \cdot \mathbf{n}(\mathbf{y}),
\end{align*}
with $\mathbf{x}, \mathbf{y} \in \Omega_S$ and $\mathbf{x} \neq \mathbf{y}$. Likewise, for the solution in $\Omega_D$, we have
\begin{align*}
u^{(*,D)}(\mathbf{x}, \mathbf{y}; \lambda) &= \Phi^{\left(\sigma_l, \sigma_t\right)}_{\lambda}(\mathbf{x}, \mathbf{y}),\\
p^{(*,D)}(\mathbf{x}, \mathbf{y}; \lambda) &= \nabla^{\left(\sigma_l,\sigma_t\right)}_{\mathbf{y}} \Phi_{\lambda}^{\left(\sigma_l,\sigma_t\right)}(\mathbf{x}, \mathbf{y}) \cdot \mathbf{n}(\mathbf{y}),
\end{align*}
with $\mathbf{x}, \mathbf{y} \in \Omega_D$ and $\mathbf{x} \neq \mathbf{y}$.

We now express the direct formulation of the electrical potential, as given in \eqref{repreuieI}, \eqref{repreuieS}, and \eqref{repreuieD}, in terms of the single- and double-layer potential operators introduced in \eqref{SP} and \eqref{DP}:
\begin{align*}
u^{(I)}(\mathbf{x}; \lambda) &= \mathcal{K}_{\lambda}^{\left(\sigma_l, \sigma_t\right)}\left[\left.u^{(I)}\right|_{\partial \Omega_I}\right](\mathbf{x}) - \mathcal{W}_{\lambda}^{\left(\sigma_l, \sigma_t\right)}\left[\left.p^{(I)}\right|_{\partial \Omega_I}\right](\mathbf{x}),\quad \mathbf{x} \in \Omega_I,\\
u^{(S)}(\mathbf{x}; \lambda) &= \mathcal{K}_{\lambda}^{\left(\sigma_t, \sigma_l\right)}\left[\left.u^{(S)}\right|_{\partial \Omega_S}\right](\mathbf{x}) - \mathcal{W}_{\lambda}^{\left(\sigma_t, \sigma_l\right)}\left[\left.p^{(S)}\right|_{\partial \Omega_S}\right](\mathbf{x}),\quad \mathbf{x} \in \Omega_S,\\
u^{(D)}(\mathbf{x}; \lambda) &= \mathcal{K}_{\lambda}^{\left(\sigma_l, \sigma_t\right)}\left[\left.u^{(D)}\right|_{\partial \Omega_D}\right](\mathbf{x}) - \mathcal{W}_{\lambda}^{\left(\sigma_l, \sigma_t\right)}\left[\left.p^{(D)}\right|_{\partial \Omega_D}\right](\mathbf{x}),\quad \mathbf{x} \in \Omega_D.
\end{align*}

Applying Proposition~\ref{propLim} by taking the limit as $\mathbf{x} \to \mathbf{z} \in \partial \Omega$, we obtain the following boundary integral equations (BIEs):
\begin{align*}
&\left(1 - c(\mathbf{z})\right)u^{(I)}(\mathbf{z}; \lambda) = \int_{\partial \Omega_I} u^{(I)}(\mathbf{y}; \lambda) p^{(*,I)}(\mathbf{z}, \mathbf{y}; \lambda)ds(\mathbf{y}) - \int_{\partial \Omega_I} p^{(I)}(\mathbf{y}; \lambda) u^{(*,I)}(\mathbf{z}, \mathbf{y}; \lambda)ds(\mathbf{y}),\; \mathbf{z} \in \partial \Omega_I,\\
&\left(1 - c(\mathbf{z})\right)u^{(S)}(\mathbf{z}; \lambda) = \int_{\partial \Omega_S} u^{(S)}(\mathbf{y}; \lambda) p^{(*,S)}(\mathbf{z}, \mathbf{y}; \lambda)ds(\mathbf{y}) - \int_{\partial \Omega_S} p^{(S)}(\mathbf{y}; \lambda) u^{(*,S)}(\mathbf{z}, \mathbf{y}; \lambda)ds(\mathbf{y}),\; \mathbf{z} \in \partial \Omega_S,\\
&\left(1 - c(\mathbf{z})\right)u^{(D)}(\mathbf{z}; \lambda) = \int_{\partial \Omega_D} u^{(D)}(\mathbf{y}; \lambda) p^{(*,D)}(\mathbf{z}, \mathbf{y}; \lambda)ds(\mathbf{y}) - \int_{\partial \Omega_D} p^{(D)}(\mathbf{y}; \lambda) u^{(*,D)}(\mathbf{z}, \mathbf{y}; \lambda)ds(\mathbf{y}),\; \mathbf{z} \in \partial \Omega_D.\\
\end{align*}

We have derived a system of three boundary integral equations (BIEs) involving six unknowns, which at this stage appear to be decoupled. However, the boundary conditions have not yet been incorporated. By imposing the conditions \eqref{BC1}, \eqref{PCC1}, and \eqref{PCC2}, we obtain an augmented system consisting of seven BIEs with seven unknowns. The resulting system is presented below.
\begin{align}
\nonumber
u^{(I)}(\mathbf{z}; \lambda) &= \int_{\Gamma_I}u^{(I)}(\mathbf{y}; \lambda)\kappa^{(I)}(\mathbf{z})p^{(*,I)}(\mathbf{z}, \mathbf{y}; \lambda)ds(\mathbf{y}) + \int_{\Gamma_{IS}}u^{(S)}(\mathbf{y}; \lambda)\kappa^{(I)}(\mathbf{z})p^{(*,I)}(\mathbf{z}, \mathbf{y}; \lambda)ds(\mathbf{y})\\ \label{BIE1}
& - \int_{\Gamma_{IS}}p^{(S)}(\mathbf{y}; \lambda)\kappa^{(I)}(\mathbf{z})u^{(*,I)}(\mathbf{z}, \mathbf{y}; \lambda)ds(\mathbf{y}),\;\; \hbox{in $\Gamma_I,$}\\[0.2cm]
\nonumber
u^{(S)}(\mathbf{z}; \lambda) &= \int_{\Gamma_I}u^{(I)}(\mathbf{y}; \lambda)\kappa^{(IS)}(\mathbf{z})p^{(*,I)}(\mathbf{z}, \mathbf{y}; \lambda)ds(\mathbf{y}) + \int_{\Gamma_{IS}}u^{(S)}(\mathbf{y}; \lambda)\kappa^{(IS)}(\mathbf{z})p^{(*,I)}(\mathbf{z}, \mathbf{y}; \lambda)ds(\mathbf{y})\\ \label{BIE4}
& - \int_{\Gamma_{IS}}p^{(S)}(\mathbf{y}; \lambda)\kappa^{(IS)}(\mathbf{z})u^{(*,I)}(\mathbf{z}, \mathbf{y}; \lambda)ds(\mathbf{y}),\;\; \hbox{in $\Gamma_{IS},$}\\[0.2cm]
\nonumber
u^{(S)}(\mathbf{z}; \lambda) &= \int_{\Gamma_S}u^{(S)}_{i,e}(\mathbf{y}; \lambda)\kappa^{(ISS)}(\mathbf{z})p^{(*,S)}(\mathbf{z}, \mathbf{y}; \lambda)ds(\mathbf{y}) + \int_{\Gamma_{IS}}u^{(S)}(\mathbf{y}; \lambda)\kappa^{(ISS)}(\mathbf{z})p^{(*,S)}(\mathbf{z}, \mathbf{y}; \lambda)ds(\mathbf{y})\\ \nonumber
& - \int_{\Gamma_{IS}}p^{(S)}(\mathbf{y}; \lambda)\kappa^{(ISS)}(\mathbf{z})u^{(*,S)}(\mathbf{z}, \mathbf{y}; \lambda)ds(\mathbf{y}) + \int_{\Gamma_{SD}}u^{(D)}(\mathbf{y}; \lambda)\kappa^{(ISS)}(\mathbf{z})p^{(*,S)}(\mathbf{z}, \mathbf{y}; \lambda)ds(\mathbf{y})\\
\label{BIE2}
& -  \int_{\Gamma_{SD}}p^{(D)}(\mathbf{y}; \lambda)u^{(*,S)}(\mathbf{z}, \mathbf{y}; \lambda)ds(\mathbf{y}),\; \; \hbox{in $\Gamma_S \cup \Gamma_{IS},$}\\[0.2cm]
\nonumber
u^{(D)}(\mathbf{z}; \lambda) &= \int_{\Gamma_{IS}}u^{(S)}(\mathbf{y}; \lambda)\kappa^{(SD)}(\mathbf{z})p^{(*,S)}(\mathbf{z}, \mathbf{y}; \lambda)ds(\mathbf{y}) - \int_{\Gamma_{IS}}p^{(S)}(\mathbf{y}; \lambda)\kappa^{(SD)}(\mathbf{z})u^{(*,S)}(\mathbf{z}, \mathbf{y}; \lambda)ds(\mathbf{y})\\
\nonumber
&+ \int_{\Gamma_{S}}u^{(S)}(\mathbf{y}; \lambda)\kappa^{(SD)}(\mathbf{z})p^{(*,S)}(\mathbf{z}, \mathbf{y}; \lambda)ds(\mathbf{y}) + \int_{\Gamma_{SD}}u^{(D)}(\mathbf{y}; \lambda)\kappa^{(SD)}(\mathbf{z})p^{(*,S)}(\mathbf{z}, \mathbf{y}; \lambda)ds(\mathbf{y})\\
&- \int_{\Gamma_{SD}}p^{(D)}(\mathbf{y}; \lambda)\kappa^{(SD)}(\mathbf{z})u^{(*,S)}(\mathbf{z}, \mathbf{y}; \lambda)ds(\mathbf{y}),\;\; \hbox{in $\Gamma_{SD},$}\\[0.2cm]
\nonumber
u^{(D)}(\mathbf{z}; \lambda) &= \int_{\Gamma_{SD}}u^{(D)}(\mathbf{y}; \lambda)\kappa^{(SDD)}(\mathbf{z})p^{(*,D)}(\mathbf{z}, \mathbf{y}; \lambda)ds(\mathbf{y}) - \int_{\Gamma_{SD}}p^{(D)}(\mathbf{y}; \lambda)\kappa^{(SDD)}(\mathbf{z})u^{(*,D)}(\mathbf{z}, \mathbf{y}; \lambda)ds(\mathbf{y})\\
&+ \int_{\Gamma_{D}}u^{(D)}(\mathbf{y}; \lambda)\kappa^{(SDD)}(\mathbf{z})p^{(*,D)}(\mathbf{z}, \mathbf{y}; \lambda)ds(\mathbf{y}),\;\; \hbox{in $\Gamma_D\cup \Gamma_{SD}.$} \label{BIE3}
\end{align}
Note that the unknowns in this system are given by
\begin{align*}
\left. u^{(I)}\right|_{\Gamma_I}, \quad \left. u^{(S)}\right|_{\Gamma_{IS}}, \quad \left. u^{(S)}\right|_{\Gamma_S}, \quad \left. u^{(D)}\right|_{\Gamma_{SD}}, \quad \left. u^{(D)}\right|_{\Gamma_D}, \quad \left. p^{(S)}\right|_{\Gamma_{IS}}, \quad \text{and} \quad \left. p^{(D)}\right|_{\Gamma_{SD}}.
\end{align*}

The system \eqref{BIE1}--\eqref{BIE3} comprises seven equations in total, since both \eqref{BIE2} and \eqref{BIE3} consist of two distinct equations.

It is also important to note that the single- and double-layer potentials appearing on the right-hand side of the equations mentioned above differ slightly from those introduced previously. In the present context, the kernels of these potentials are modified by multiplication with functions of the variable $\mathbf{z}$ that are constant except at a finite number of points. Specifically, in \eqref{BIE1}, the corresponding kernels are multiplied by
\begin{align*}
\kappa^{(I)}(\mathbf{z}) = \left\{
\begin{array}{cl}
\frac{2}{3}, & \text{if $\mathbf{z} \in \Gamma_I \setminus \left\{(0,0), (0,b)\right\}$},\\[0.2cm]
\frac{4}{5}, & \text{if $\mathbf{z} = (0,0)$ or $\mathbf{z} = (0,b)$}.
\end{array}
\right.
\end{align*}
which accounts for geometric singularities at the endpoints of the boundary.

Accordingly, in each of the following equations, the kernels are multiplied by the corresponding functions:
\begin{align*}
\kappa^{(IS)}(\mathbf{z}) = \left\{
\begin{array}{cc}
\frac{2}{3}, & \text{in $\Gamma_{IS} \setminus \left\{\left(\frac{a}{2} - h,0\right), \left(\frac{a}{2} - h,b\right)\right\}$},\\[0.2cm]
\frac{4}{5}, & \text{$\mathbf{z} = \left(\frac{a}{2} - h,0\right)$ or $\mathbf{z} = \left(\frac{a}{2} - h,b\right)$}, 
\end{array}
\right.
\end{align*}
\begin{align*}
\kappa^{(ISS)}(\mathbf{z}) = \left\{
\begin{array}{cc}
\frac{2}{3}, & \text{in $\Gamma_{IS}\cup\Gamma_S \setminus \left\{(\frac{a}{2} - h,0), (\frac{a}{2} - h,b)\right\}$},\\
\frac{4}{5}, & \text{$\mathbf{z} = (\frac{a}{2} - h,0)$ or $\mathbf{z} = (\frac{a}{2} - h,b)$}, 
\end{array}
\right.
\end{align*}
\begin{align*}
\kappa^{(SD)}(\mathbf{z}) = \left\{
\begin{array}{cc}
\frac{2}{3}, & \text{in $\Gamma_{SD} \setminus \left\{(\frac{a}{2} + h,0), (\frac{a}{2} + h,b)\right\}$},\\
\frac{4}{5}, & \text{$\mathbf{z} = (\frac{a}{2} + h,0)$ or $\mathbf{z} = (\frac{a}{2} + h,b)$}, 
\end{array}
\right.
\end{align*}
and 
\begin{align*}
\kappa^{(SDD)}(\mathbf{z}) = \left\{
\begin{array}{cc}
\frac{2}{3}, & \text{in $\Gamma_{SD}\cup \Gamma_{D} \setminus \left\{(a,0), (a,b)\right\}$},\\
\frac{4}{5}, & \text{$\mathbf{z} = (a,0)$ or $\mathbf{z} = (a,b)$}.
\end{array}
\right.
\end{align*}
Since these functions are bounded, the properties of the single- and double-layer potentials established in Section~\ref{BIEsect} remain valid for the corresponding modified kernels. Moreover, the system of equations \eqref{BIE1}-\eqref{BIE3} involves only nonsingular integrals. In particular, equation \eqref{BIE1} contains the second and third integrals, both of which are nonsingular.

\section{The numerical resolution of the BIE system}
In this section, we apply a numerical method to solve the system of singular boundary integral equations \eqref{BIE1}-\eqref{BIE3}. Our approach consists in constructing a scheme to discretize the singular integrals appearing in the system, thereby transforming it into a system of linear algebraic equations. Within this framework, the eigenvalues are approximated as those for which the resulting system admits a nontrivial solution. It is worth noting that, due to the geometry of the curves involved in the integrals, explicit expressions for the associated kernels can be derived in terms of elementary functions. Accordingly, we present tables containing the values of these kernels, whose construction is described below.

\subsection{Summary of the Numerical Approximation of the BIE System}

The resulting system of boundary integral equations (BIEs) involves integrals with singular kernels defined in terms of the fundamental solutions of the governing equation. Owing to the rectangular geometry of the domain and the particular structure of these functions, explicit expressions for the kernels can be obtained in terms of the Bessel functions of the second kind $N_0(x)$ and $N_1(x)$, which exhibit logarithmic singularities at the origin. This behavior requires special numerical treatment during the discretization process.

To solve the problem computationally, a uniform discretization of the boundaries and interfaces is performed by partitioning them into segments. Quadrature formulas adapted to weakly singular integrals are then applied, thereby transforming the continuous system into a linear algebraic system. The existence of nontrivial solutions for certain values of the spectral parameter $\lambda$ allows for the identification of the eigenvalues of the bidomain operator.

Since the kernels involve Bessel functions with singular behavior, their evaluation is carried out using truncated versions obtained from their series expansions. In particular, a fixed approximation is implemented by retaining the first five terms of the series. The implementation, developed in Julia, is available at \href{https://colab.research.google.com/drive/1ttURdbUb1CUx-bTgX8Mo7fYB0RarFR1e?usp=sharing}{Julia code}, and allows for adjusting the number of terms to improve accuracy.

The resulting kernel values, evaluated between pairs of discretized points, are stored in reusable tables, thereby facilitating the assembly of the system for different values of $\lambda$. The eigenvalues are determined by analyzing the loss of rank of the resulting system, either through the computation of the smallest singular value or by detecting zeros of the determinant of the discretized system.

\subsection{Discretization Scheme and Kernel Evaluation}

The discretization of the BIE system is carried out on the external boundaries and the interfaces between the different subregions. To this end, the curves $\Gamma_I$, $\Gamma_S$, $\Gamma_D$, $\Gamma_{IS}$, and $\Gamma_{SD}$ are uniformly partitioned into segments, generating nodes $\{ \mathbf{x}_j \}_{j=1}^{N}$ at which both the densities and the kernels are evaluated.

At points where the integration variable coincides with the evaluation point, specialized quadrature rules are employed to handle the weak logarithmic singularity. The kernels $p^{(*,\cdot)}$ and $u^{(*,\cdot)}$ are computed explicitly in terms of $N_0(x)$ and $N_1(x)$ for each pair $(\mathbf{z}, \mathbf{y})$.

In addition, an internal discretization of the domain $\Omega$ is performed to evaluate $u_{i,e}(\mathbf{x})$ in the interior, using the discretized integral representations. This allows for the reconstruction and visualization of the eigenmodes, as well as for the assessment of the accuracy of the solution through comparison with analytical solutions or more refined simulations.

The eigenvalues $\lambda$ are determined by examining when the linear system loses rank, which is identified either by computing the smallest singular value or by locating the zeros of the determinant of the system.

\section{Constants, Resulting System, and Numerical Validation}

During the formulation and implementation, the following constants are employed:

\begin{itemize}
    \item $a = 4.5 \times 10^{-15},\quad b = 2.5 \times 10^{-15}$: dimensions of the total rectangular domain $\Omega = (0, a) \times (0, b)$.
    
    \item $h = a \cdot 0.1 = 4.5 \times 10^{-16}$: semi-width of the septum defining the interfaces $\Gamma_{IS}$ and $\Gamma_{SD}$.
    
    \item $\sigma_{il} = 3.0,\quad \sigma_{it} = 1.0,\quad \sigma_{el} = 2.0,\quad \sigma_{et} = 1.65$: longitudinal and transverse intra and extracellular conductivities.
    
    \item $\epsilon_c = h \cdot 10^{-7} = 4.5 \times 10^{-23}$: regularization parameter for handling the kernel singularity.
    
    \item $\gamma \approx 0.5772$: Euler's constant in the analytical expressions of $N_0(x)$ and $N_1(x)$.
    
    \item $nnodes\_curve = 5$: number of points per discretization curve.
    
    \item $nnodes\_plot = 20$: number of points for result visualization.
\end{itemize}

The resulting linear system after discretization has the following characteristics:

\begin{itemize}
    \item It is large in size due to the separate treatment of each subdomain and its respective boundaries.
    
    \item It is structured in blocks, reflecting the interaction between regions and the conditions at interfaces.
    
    \item It depends on the spectral parameter $\lambda$, requiring exploration for different values.
    
    \item It is generally not symmetric due to the nature of the double layers and the coupling conditions.
\end{itemize}

\subsection{Numerical Example}

\begin{figure}[H]
    \centering
    \parbox{0.40\textwidth}{
        \centering
        \includegraphics[scale=0.25]{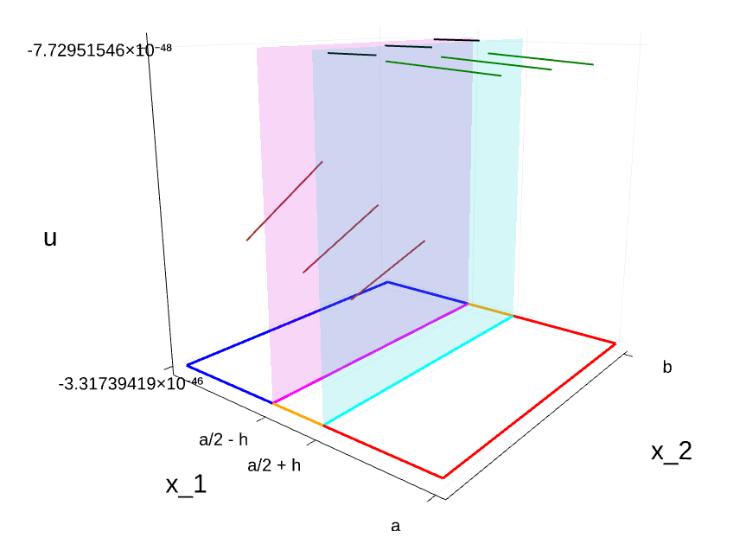}\\
    }
    \hfill
    \parbox{0.40\textwidth}{
        \centering
        \includegraphics[scale=0.33]{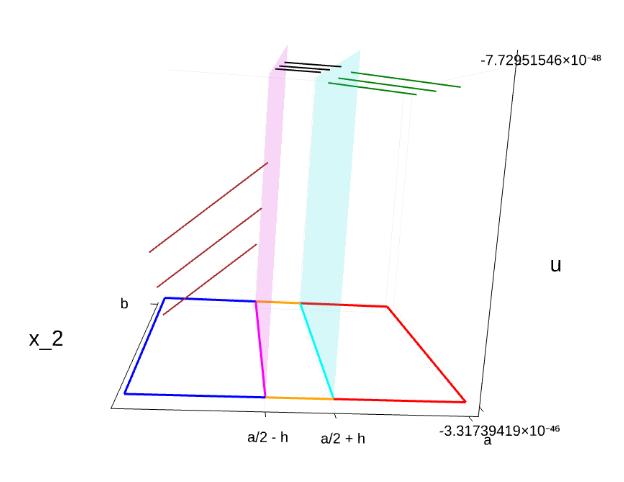}\\
    }
    \hfill
    \parbox{0.18\textwidth}{
        \centering
        \includegraphics[scale=0.20]{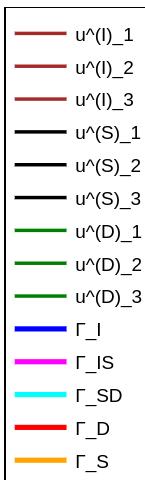}\\
    }
    \caption{Numerical representation of the solutions $u^{(I)}$, $u^{(S)}$, and $u^{(D)}$ within the region.}
    \label{fig:u_potentials}
\end{figure}

Figure~\ref{fig:u_potentials} illustrates the evolution of the electrical potential across the left ventricle, septum, and right ventricle, computed using the numerical method described in the previous sections, as the variable $\mathbf{z} = (x, y)$ moves along a straight line parallel to the $x$-axis. In this sense, the figure represents the corresponding regions in the $xy$-plane, while the value of the electrical potential is depicted along the direction orthogonal to this plane.

The numerically approximated values of the functions $u^{(I)}$, $u^{(S)}$, and $u^{(D)}$ obtained through the proposed integral scheme exhibit behavior consistent with the expected qualitative properties of the model. In particular, the solutions remain bounded and display a coherent structure across the different subregions.

The very small magnitude of the computed values, on the order of $10^{-46}$ to $10^{-47}$, is related to the scaling of the problem and the geometric configuration considered. In the absence of external stimulation or forcing terms, the model predicts low-amplitude potentials, which is consistent with the passive regime under consideration.

\section{Conclusions}

At this point, we focus on highlighting the qualitative behavior exhibited by the computed potentials. Note that, even though perfect contacto conditions for the current have been imposed at the interfaces between the subregions composing $\Omega$, the electrical potential exhibits a jump across the interface. This behavior may lead to a change in the propagation speed of the wavefront when the full diffusion problem is considered. 

From a theoretical perspective, in this work we have used potential theory to reduce an anisotropic elliptic boundary value problem, subject to boundary conditions and perfect transmission conditions, to a system of seven singular boundary integral equations. This reduction lowers the dimensionality of the problem and transforms a boundary value problem (PDE) into an integral formulation involving compact operators. 

% --- Bibliografía ---
\printbibliography

@Book{Ortner2015,
  author    = {Norbert Ortner and Peter Wagner},
  title     = {Fundamental Solutions of Linear Partial Differential Operators. Theory and Practice},
  year      = {2015},
  publisher = {Springer},
  address = {London},
}

@Book{Bell1968,
  author    = {William W. Bell},
  title     = {Special functions for and engineers},
  year      = {1968},
  publisher = {Van Nostrand},
  address   = {London},
}

@Book{Korenev2002,
  author    = {Boris G. Korenev},
  publisher = {Chapman \& Hall/CRC},
  title     = {Bessel functions and their applications},
  year      = {2002},
  address   = {Boca Raton},
}

@Book{Colton2013,
  author    = {David Colton and Rainer Kress},
  publisher = {SIAM},
  title     = {Integral equation methods in scattering theory},
  year      = {2013},
  address   = {Philadelphia},
}

@Book{Brebbia1984,
  author    = {C. A. Brebbia and J. C. F. Telles and L. C. Wrobel},
  publisher = {Springer},
  title     = {Boundary Element Techniques},
  year      = {1984},
  address = {Berlin},
}

@article{Tung1978,
    author = {L. Tung},
    title = {A bidomain model for describing ischemic myocardial D-C potentials.},
    journal = {Ph.D Ph.D. dissertation, Massachusetts Inst. Technol,},
    year = {1978}
}

@book{Sundnes2006,
    author = {J. Sundnes and G. T.Lines and X. Cai and B.F. Nielsen and K.A. Mardal and A. Tveito},
    title = {Computing the Electrical Activity in the Heart. Monographs in Computational Science and Engineering},
    publisher = {Springer, Berlin Heidelberg},
    year = {2006}
}

@inbook{Colli-Franzone2002,
    author = {P. Colli-Franzone and G. Savaré},
    title = {Evolution Equations, Semigroups and Functional Analysis: In Memory of Brunello Terreni},
    publisher = {Evolution Equations, Semigroups and Functional Analysis},
    year = {2002},
    chapter = {Degenerate Evolution Systems Modeling the Cardiac Electric Field at Micro- and Macroscopic Level}
}

@article{Veneroni2009,
    author = {M. Veneroni},
    title = {Reaction-diffusion systems for the macroscopic bidomain model of the cardiac electric field},
    journal = {Nonlinear Anal., Real World Appl.},
    year = {2009}
}

@article{Bourgault2009,
    author = {Y. Bourgault and Y. Coudière and C. Pierre},
    title = {Existence and uniqueness of the solution for the bidomain model used in cardiac electrophysiology.},
    journal = {Nonlinear Anal., Real World Appl.},
    year = {2009}
}

@article{Felipe-Sosa2022,
    author = {A. Fraguela and R. Felipe-Sosa and J. Henry and M. F. Márquez},
    title = {Existence of a T-Periodic solution for the monodomain model corresponding to an isolated ventricle due to ionic-diffusive relations},
    journal = {Acta Applicandae Mathematicae},
    year = {2022}
}

@article{Felipe-Sosa2023,
    author = {R. Felipe-Sosa and A. Fraguela-Collar and Y. H. García-Gómez},
    title = {On the strong convergence of the Faedo-Galerkin approximations to a strong t-periodic solution of the torso-coupled bidomain model},
    journal = {Math. Model. Nat. Phenom.},
    year = {2023}
}

@article{Boulakia2010,
    author = {M. Boulakia and S. Cazeau and M. A. Fernández and J. F. Gerbeau},
    title = {Mathematical modeling of electriocardiograms: A numerical study},
    journal = {Annals of Biomedical Engineering},
    year = {2010}
}

@article{Medrano2002,
    author = {G. A Medrano and A. de Micheli and A. Aranda, P. Iturralde, R. Chávez Domínguez},
    title = {An experimental contribution to the concept of “Jumping wave” phenomenon in the interventricular septum},
    journal = {Archives of Cardiology of Mexico},
    year = {2002}
}

@article{Medrano1957,
    author = {G. Medrano and A. Bisteni and R. W. Brancato and F. Pileggi and D. Sodi-Pallares},
    title = {The activation of the interventricular septum in the dog's heart under normal conditions and in bundle-branch block},
    journal = {Ann N Y Acad Sci},
    year = {1957}
}

@article{Colli-Franzone2006,
    author = {P. Colli Franzone and L.F. Pavarino and B. Taccardi} ,
    title = {Simulating patterns of excitation, repolarization and action potential duration with cardiac Bidomain and Monodomain models},
    journal = {Mathematical Biosciences},
    year = {2006}
}

@article{Patelli2017,
    author = {Alessandro S. Patelli and Luca Dedè and Toni Lassila and Andrea Bartezzaghi and Alfio Quarteroni},
    title = {Isogeometric approximation of cardiac electrophysiology models on surfaces: An accuracy study with application to the human left atrium},
    journal = {Computer Methods in Applied Mechanics and Engineering},
    year = {2017}
}

@article{Bucelli2021,
    author = {Michele Bucelli and Matteo Salvador and Luca Dede and Alfio Quarteroni},
    title = {Multipatch Isogeometric Analysis for electrophysiology: Simulation in a humanheart},
    journal = {Computer Methods in Applied Mechanics and Engineering},
    year = {2021}
}

@article{Charawi2017,
    author = {Lara Antonella Charawi},
    title = {Isogeometric overlapping Schwarz preconditioners for the Bidomain reaction–diffusion system},
    journal = {Computer Methods in Applied Mechanics and Engineering},
    year = {2017}
}

@book{Linkov2002,
    author = {A.M.Linkov},
    title = {Boundary Integral Equations in Elasticity Theory},
    publisher = {Springer},
    year = {2002}
}

@article{Kotadia2020,
    author = {Kotadia I and Whitaker J and Roney C and Niederer S and O'Neill M and Bishop M and Wright M.},
    title = {Anisotropic Cardiac Conduction},
    journal = {Arrhythm Electrophysiol Rev.},
    year = {2020}
}

\end{document}